\RequirePackage{plautopatch}
\documentclass{amsart}

\usepackage{amsmath,amssymb,amsthm,mathtools,mathrsfs}
\usepackage{booktabs}
\usepackage{array}
\usepackage{enumitem}
\usepackage{float}
\usepackage{iftex}
\ifPDFTeX\else

\fi
\usepackage{tikz}
\usetikzlibrary{arrows.meta}
\ifPDFTeX
\usepackage[hidelinks]{hyperref}
\else
\usepackage[dvipdfmx,hidelinks]{hyperref}
\fi
\usepackage{cleveref}

\numberwithin{equation}{section}

\theoremstyle{plain}
\newtheorem{theorem}{Theorem}[section]
\newtheorem{proposition}[theorem]{Proposition}
\newtheorem{lemma}[theorem]{Lemma}
\newtheorem{corollary}[theorem]{Corollary}
\theoremstyle{definition}
\newtheorem{definition}[theorem]{Definition}
\newtheorem*{definition*}{Definition}
\theoremstyle{remark}
\newtheorem{remark}[theorem]{Remark}

\newcommand{\R}{\mathbb R}
\newcommand{\B}{\mathbb B}
\newcommand{\BetaBall}{\mathrm F\mathbb B^n_{\beta_n}}
\newcommand{\Pyr}{\mathcal P}
\newcommand{\Lipplus}{\operatorname{Lip}^{+}_{1}}
\newcommand{\Var}{\operatorname{Var}}
\DeclareMathOperator{\supp}{supp}
\DeclareMathOperator{\arcosh}{arcosh}
\DeclareMathOperator{\arsinh}{arsinh}
\DeclareMathOperator{\artanh}{artanh}
\DeclareMathOperator{\Exp}{Exp}
\newcommand{\dF}{d_{\mathrm F}}
\newcommand{\dK}{d_{\mathrm K}}
\newcommand{\eps}{\varepsilon}

\title[Funk Beta Balls]{Funk Beta Balls: Exact Potentials and a Three-Phase Pyramid Diagram}
\author{Shigeaki Yokota}
\date{}

\begin{document}

\begin{abstract}
In a high-dimensional Funk ball, the endpoint potential makes the forward and reverse distances behave differently, while symmetrization erases the distinction that survives in pyramid limits. For Euclidean beta-type radial measures, we determine the natural-scale limits in all three parameter phases. The divergent phase yields directed Gaussian pyramids indexed by the balance between dimension and radial concentration. A positive limiting parameter yields a Gaussian--chi Funk horocone pyramid, built from Gaussian bases, chi-distributed heights, and a directed endpoint-potential increment, while a vanishing parameter yields the universal directed star pyramid. In the last phase, the dimension-dependent logarithmic contributions cancel exactly, so the conclusion requires no further rate condition.
\end{abstract}

\maketitle
\tableofcontents

\section{Introduction}\label{sec:introduction}

On high-dimensional unit spheres, normalized Riemannian volume forces every $1$-Lipschitz observable to be nearly constant on almost all of the sphere.  Gromov developed the geometry of mm-spaces to treat this phenomenon while retaining both the distance and the measure: the distance specifies the $1$-Lipschitz observables, whereas the measure decides whether they are nearly constant \cite[Introduction]{shioya2016mmg}.  Gromov introduced the observable distance and the stronger box distance to express concentration as convergence, and proposed a natural compactification \cite[Introduction and Section~2.2]{esaki-kazukawa-mitsuishi2024cones}.  In the later pyramid formulation, each mm-space is sent to the box-closed family of all spaces that it dominates \cite[Definition~2.24]{esaki-kazukawa-mitsuishi2024cones}.  Shioya metrized this topology and made the space of pyramids compact \cite[Definition~4.5 and Theorem~4.6]{shioya2022sugaku}.  This paper asks for a directed counterpart in the Funk geometry of the Euclidean unit ball with beta-type radial measures.  The forward and reverse behavior on the ray below makes asymmetry unavoidable.

In the forward Funk ball, the endpoint potential $\phi(x)\coloneqq-\frac12\log(1-|x|^2)$ is strictly subcritical at every point: its norm dual to the Klein distance $\dK$ is $|x|<1$, and the exact identity
\[ \dF(x,y)=\dK(x,y)+\phi(y)-\phi(x) \]
gives a positive distance between distinct points.  Its supremum over the ball is nevertheless one.  There is therefore no uniform subcritical margin and no uniform control of the reverse distance, and one orientation of the distance can vanish in a limit.

This imbalance is already visible on one ray.  For $0<r<1$ and $\theta\in S^{n-1}$,
\begin{equation}\label{eq:origin-distances} \dF(0,r\theta)=-\log(1-r)\longrightarrow+\infty,\qquad \dF(r\theta,0)=\log(1+r)\le\log2. \end{equation}

For $n\ge1$ and $b>0$, define a probability measure on the Euclidean unit ball $\B^n$ by
\begin{equation}\label{eq:beta-measure} d\mu_{n,b}(x) \coloneqq\frac{\Gamma(n/2+b)}{\pi^{n/2}\Gamma(b)}(1-|x|^2)^{b-1}\,dx. \end{equation}
Given a positive sequence $(\beta_n)$, put
\[ \BetaBall\coloneqq(\B^n,\dF,\mu_{n,\beta_n}). \]
The upright tag $\mathrm F$ records the Funk distance.
A quasi-metric measure space (qm-space) carries a directed distance whose max symmetrization is a complete separable metric, together with a full-support Borel probability measure.  For a qm-space $X$, write $\Pyr(X)$ for the Box-closed family of its exact factors.  Weak convergence of pyramids means sequential Painlev\'e--Kuratowski convergence, also called weak Hausdorff convergence in mm-space theory.  The exact potential makes the natural scales depend on whether $\beta_n$ diverges, approaches a positive constant, or tends to zero.

The symmetrization $\dF^{\mathrm s}(x,y)=\dK(x,y)+|\phi(y)-\phi(x)|$ replaces the signed increment by its absolute value, so it forgets which endpoint carries the potential and cannot distinguish the three limits below.  Comparison geometry and analysis on irreversible metric spaces are developed in \cite{ohta2021comparison,kristaly-zhao2022geometry,kristaly-ohta-zhao2025analysis}.

The three target pyramids are denoted by $\mathcal G_q$, $\mathcal H_\beta$, and $\mathcal S_{\rhd}$ for the directed Gaussian, Funk horocone, and directed star cases, respectively.  The fixed-parameter horocone carries the chi height law $\chi_{2\beta}$.  The main theorem, \Cref{thm:main}, gives the following weak limits.  If $\beta_n\to+\infty$ and $n/(n+2\beta_n)\to q^2$ for some $q\in[0,1]$, then $\Pyr(\sqrt{2\beta_n}\,\BetaBall)\to\mathcal G_q$.  If $\beta_n\to\beta\in(0,+\infty)$, then $\Pyr(\BetaBall)\to\mathcal H_\beta$.  If $\beta_n\to0$, then $\Pyr(2\beta_n\BetaBall)\to\mathcal S_{\rhd}$, without any assumption on $\beta_n\log n$.  \Cref{tab:phase-mechanism} records the scales and geometric components of these limits.

\begin{table}[ht]
\centering
\small
\setlength{\tabcolsep}{2pt}
\renewcommand{\arraystretch}{1.15}
\begin{tabular}{@{}p{0.18\textwidth}p{0.10\textwidth}p{0.14\textwidth}p{0.13\textwidth}p{0.15\textwidth}p{0.17\textwidth}@{}}
\toprule
\raggedright Regime & \raggedright Natural scale & \raggedright Radial survival & \raggedright Angular limit & \raggedright Endpoint potential & \raggedright Target \tabularnewline
\midrule
\raggedright $\beta_n\to+\infty$, $n/(n+2\beta_n)\to q^2$
& \raggedright $\sqrt{2\beta_n}$
& \raggedright Gaussian fluctuation
& \raggedright Gaussian
& \raggedright linear radial potential
& \raggedright directed Gaussian $\mathcal G_q$ \tabularnewline
\raggedright $\beta_n\to\beta\in(0,+\infty)$
& \raggedright $1$
& \raggedright chi height $\chi_{2\beta}$
& \raggedright Gaussian
& \raggedright full log-height potential
& \raggedright Funk horocone $\mathcal H_\beta$ \tabularnewline
\raggedright $\beta_n\to0$; no condition on $\beta_n\log n$
& \raggedright $2\beta_n$
& \raggedright exponential log-height
& \raggedright discrete rays
& \raggedright critical cancellation
& \raggedright directed star $\mathcal S_{\rhd}$ \tabularnewline
\bottomrule
\end{tabular}
\caption{The common radial--angular--potential mechanism in the three phases.}
\label{tab:phase-mechanism}
\end{table}

Under the radial representation of \Cref{thm:gamma-master,thm:horospherical-transfer}, height renormalization removes the dimensional logarithmic term and leaves the regime-dependent radial component.  At the small-beta scale, this surviving component has an exponential limit.

The truncated-kernel, horosphere, and finite-profile comparisons establish the three clauses in this order, after which bounded-measurement invariants identify the parameters and separate the three targets.

\section{Preliminaries}\label{sec:preliminaries}

\subsection{Upstream dependencies}

The directed compactification used in the limit arguments is inherited from the upstream theory.  We record the precise notions and stability results needed here together with their sources.  Two further tools from \cite{one-sided-pyramids}, an approximate one-sided extension and a high-mass domination principle, are used in the transfer arguments collected in \Cref{app:transfer-tools}, where their precise forms and sources are recorded.

\begin{definition*}[Qm-spaces {\cite[Definition~2.2]{gds3-ja}}]
A triple $(X,d_X,\mu_X)$, or simply $X$, is a \emph{qm-space} if $d_X\colon X\times X\to[0,+\infty)$ satisfies $d_X(x,x)=0$ and the triangle inequality, its max symmetrization
\[ d_X^{\mathrm s}(x,y)\coloneqq\max\{d_X(x,y),d_X(y,x)\} \]
is a complete separable metric, and $\mu_X$ is a Borel probability measure on $(X,d_X^{\mathrm s})$ with full support.  The term qm-space abbreviates quasi-metric measure space.
\end{definition*}

\begin{definition*}[Isomorphism of qm-spaces {\cite[Definition~2.3]{gds3-ja}}]
Two qm-spaces are isomorphic if a measure-preserving bijection between them preserves the directed distance in both variables.  All collections of qm-spaces below consist of isomorphism classes.
\end{definition*}

\begin{definition*}[One-sided observables {\cite[Definition~2.5]{gds3-ja}}]
The one-sided observables of $X$ are
\[ \Lipplus(X,d_X)\coloneqq\{f\colon X\to\R\mid f(y)-f(x)\le d_X(x,y)\text{ for all }x,y\in X\}. \]
\end{definition*}
This is the semi-Lipschitz function class studied by Romaguera and Sanchis \cite{romaguera2000semi} in the quasi-metric setting.

\begin{definition*}[Exact domination {\cite[Definition in \S2.1]{gds3-ja}}]
A Borel map $p\colon X\to Y$ is an \emph{exact domination map} if $p_*\mu_X=\mu_Y$ and
\[ d_Y(p(x),p(y))\le d_X(x,y)\qquad(x,y\in X). \]
We write $Y\preceq X$ when such a map exists.  Stojmirovi\'c~\cite{stojmirovic2004quasi} calls a related structure a pq-space, but its axioms are not identical to those used here.
\end{definition*}

For a closed set $S\subset X\times Y$ and real-valued functions $h,k$ on $X\times Y$, put
\[ d_\infty^S(h,k)\coloneqq\sup_{(x,y)\in S}|h(x,y)-k(x,y)|, \]
with value zero when $S$ is empty, and let $(d_\infty^S)_H$ be the induced Hausdorff pseudometric on function families.  Write $\mathcal T(\mu_X,\mu_Y)$ for the set of couplings of $\mu_X$ and $\mu_Y$.

\begin{definition*}[Box distance {\cite[Definitions~2.7 and~2.11]{gds3-ja}}]
The \emph{Box distance} between qm-spaces is
\begin{equation}\label{eq:qm-box-distance}
\begin{aligned}
\Box(X,Y)\coloneqq\inf\bigl\{&\max\{1-\pi(S),2(d_\infty^S)_H(\Lipplus(X,d_X)\circ\operatorname{pr}_1,\Lipplus(Y,d_Y)\circ\operatorname{pr}_2)\}\ \bigm|\\
&\pi\in\mathcal T(\mu_X,\mu_Y),\ S\subset X\times Y\text{ is closed}\bigr\}.
\end{aligned}
\end{equation}
This is the observable representation of the directed Box distance.
\end{definition*}

Exact domination is closed under this distance: if $Y_n\preceq X_n$, $\Box(X_n,X)\to0$, and $\Box(Y_n,Y)\to0$, then
\begin{equation}\label{eq:domination-box-closedness} Y\preceq X. \end{equation}
Indeed, this is the qm-space specialization of the additive-error stability and distortion identification in \cite[Definition~3.1 and Propositions~3.4, 3.5, and~3.8]{one-sided-pyramids}.

\begin{definition*}[Pyramids {\cite[Definition~A.2]{gds3-ja}}]
A nonempty set $\mathcal P$ of qm-spaces is a \emph{pyramid} if it is downward closed under $\preceq$, every two members have a common dominator in $\mathcal P$, and $\mathcal P$ is closed with respect to $\Box$.  For a qm-space $X$, its associated pyramid \cite[Proposition~A.9]{gds3-ja} is
\[ \Pyr(X)\coloneqq\{Y\colon Y\preceq X\}. \]
It is a pyramid: nonemptiness, downward closedness, and directedness follow from the definition, while Box closedness follows from \eqref{eq:domination-box-closedness}.
\end{definition*}

\begin{definition*}[Weak convergence of pyramids {\cite[Definition~A.3, Theorem~5.1, and Corollary~5.3]{gds3-ja}}]
A sequence of pyramids $\mathcal P_n$ \emph{converges weakly} to a pyramid $\mathcal P$ if it converges as a sequence of closed subsets of the Box-distance space in the sense of sequential Painlev\'e--Kuratowski convergence, also called weak Hausdorff convergence in mm-space theory.  Explicitly,
\begin{align}
\Box(A,\mathcal P_n)&\longrightarrow0 &&(A\in\mathcal P), \label{eq:weak-pyramid-inner}\\
\liminf_{n\to\infty}\Box(A,\mathcal P_n)&>0 &&(A\notin\mathcal P). \label{eq:weak-pyramid-outer}
\end{align}
Equivalently, \eqref{eq:weak-pyramid-outer} says that whenever $A_k\in\mathcal P_{n(k)}$ and $A_k\to A$ in Box distance along a subsequence, then $A\in\mathcal P$.  Every sequence of qm-space pyramids has a subsequence converging weakly to a pyramid.
\end{definition*}

The symbol $cX$ means that the directed distance of $X$ is multiplied by $c>0$.

For $N\ge1$ and $u,v\in\R^N$, put
\begin{equation}\label{eq:directed-cube-distance} d_N^+(u,v)\coloneqq\max_{1\le j\le N}(v_j-u_j)_+. \end{equation}
\begin{definition*}[Bounded measurements {\cite[Definition~4.1]{gds3-ja}}]
For $R>0$, the \emph{bounded measurement set} $\mathcal M(X;N,R)$ consists of all $F_*\mu_X$ for Borel maps $F\colon X\to[-R,R]^N$ satisfying
\begin{equation}\label{eq:bounded-measurement-condition} d_N^+(F(x),F(y))\le d_X(x,y) \qquad(x,y\in X). \end{equation}
For a pyramid $\mathcal P$, define
\begin{equation}\label{eq:pyramid-bounded-measurement} \mathcal M(\mathcal P;N,R)\coloneqq\bigcup_{X\in\mathcal P}\mathcal M(X;N,R). \end{equation}
This set is compact in the Prokhorov distance $d_{\mathrm P}$ induced by the $\ell^\infty$ metric on $[-R,R]^N$.
\end{definition*}

\begin{definition*}[Directed coordinate models]
If $\nu$ is a probability measure on this cube, put
\[ \mathsf B_N^+(\nu)\coloneqq(\supp\nu,d_N^+,\nu). \]
Then
\begin{equation}\label{eq:coordinate-model-membership} \nu\in\mathcal M(\mathcal P;N,R)\quad\Longleftrightarrow\quad\mathsf B_N^+(\nu)\in\mathcal P, \end{equation}
because a measurement map is an exact domination map onto its support and the coordinate projections of $\mathsf B_N^+(\nu)$ recover $\nu$.  Moreover, the qm-space specialization of \cite[Lemma~B.1]{gds3-ja} gives
\begin{equation}\label{eq:coordinate-model-box} \Box(\mathsf B_N^+(\mu),\mathsf B_N^+(\nu))\le2d_{\mathrm P}(\mu,\nu). \end{equation}
The factor two follows by coupling $\mu$ and $\nu$ so that, outside mass at most $\eps$, paired cube points have $\ell^\infty$ distance at most $\eps$; the induced relation has directed-distance distortion at most $2\eps$.
\end{definition*}

Bounded measurements are monotone under exact domination and stable under Box approximation:
\begin{equation}\label{eq:measurement-box-stability}
\begin{aligned}
Y\preceq X&\Longrightarrow\mathcal M(Y;N,R)\subset\mathcal M(X;N,R),\\
(d_{\mathrm P})_H(\mathcal M(X;N,R),\mathcal M(Y;N,R))&\le\Box(X,Y).
\end{aligned}
\end{equation}
The first assertion follows by pulling measurements back along a domination map.  The second is the two-sided consequence of the finite-measurement inclusion in \cite[Proposition~3.9]{one-sided-pyramids}; see also \cite[Lemma~B.2]{gds3-ja}.

For pyramids $\mathcal P_n,\mathcal P$, weak convergence $\mathcal P_n\to\mathcal P$ is equivalent to the following condition for every $N\ge1$ and $R>0$:
\begin{equation}\label{eq:measurement-convergence-criterion} (d_{\mathrm P})_H(\mathcal M(\mathcal P_n;N,R),\mathcal M(\mathcal P;N,R))\longrightarrow0. \end{equation}
For a qm-space $X$, the corresponding membership criterion says that $X\in\mathcal P$ is equivalent to
\begin{equation}\label{eq:measurement-membership-criterion} \mathcal M(X;N,R)\subset\mathcal M(\mathcal P;N,R) \end{equation}
for every $N\ge1$ and $R>0$.
These facts, including the compactness after \eqref{eq:pyramid-bounded-measurement}, are the qm-space specializations of \cite[Theorem~4.4, Theorem~B.5, and Proposition~B.6]{gds3-ja} and \cite[Lemmas~3.12 and~3.14]{one-sided-pyramids}.

\subsection{Standing notation}

The notation fixes the finite-dimensional geometry and radial variables used in the limit arguments.

For $x\in\B^n$ and $v\in\R^n$, define the Klein and forward Funk norms by
\begin{align}
 \alpha_x(v)
 &\coloneqq\frac{\sqrt{(1-|x|^2)|v|^2+\langle x,v\rangle^2}}{1-|x|^2},
 \label{eq:klein-norm}\\
 F_x(v)
 &\coloneqq\frac{\sqrt{(1-|x|^2)|v|^2+\langle x,v\rangle^2}+\langle x,v\rangle}{1-|x|^2}.
 \label{eq:funk-norm}
\end{align}
Their distances are denoted by $\dK$ and $\dF$.
For each $n$, let
\[ U_n\sim\Gamma(n/2,1), \qquad V_n\sim\Gamma(\beta_n,1), \qquad \Theta_n\sim\sigma_{n-1} \]
be independent, and put
\[ \mathbf X_n\coloneqq\sqrt{\frac{U_n}{U_n+V_n}}\,\Theta_n, \qquad \mathbf R_n\coloneqq\artanh|\mathbf X_n|. \]

\subsection{Exact finite-dimensional geometry}

\begin{theorem}[Exact potential and finite-dimensional structure]\label{thm:exact-finite-geometry}
For all $x,y\in\B^n$,
\begin{equation}\label{eq:dual-norm} \lVert(d\phi)_x\rVert_{\alpha^*}=|x|<1 \end{equation}
and
\begin{equation}\label{eq:distance-exact} \dF(x,y)=\dK(x,y)+\phi(y)-\phi(x). \end{equation}
Moreover,
\begin{equation}\label{eq:symmetrization} \dF^{\mathrm s}(x,y)=\dK(x,y)+|\phi(y)-\phi(x)|. \end{equation}
For every $b>0$, the triple $(\B^n,\dF,\mu_{n,b})$ is a qm-space, and
\begin{equation}\label{eq:affine-lip} \Lipplus(\B^n,\dF) =\phi+\operatorname{Lip}_1(\B^n,\dK). \end{equation}
\end{theorem}

\begin{proof}
Since
\[ (d\phi)_x(v)=\frac{\langle x,v\rangle}{1-|x|^2}, \]
the two norms in \eqref{eq:klein-norm} and \eqref{eq:funk-norm} satisfy $F_x=\alpha_x+(d\phi)_x$.  The Klein metric tensor and its inverse are
\[ G_x=\frac{(1-|x|^2)I+xx^{\mathsf T}}{(1-|x|^2)^2}, \qquad G_x^{-1}=(1-|x|^2)(I-xx^{\mathsf T}). \]
Substitution of the covector $x/(1-|x|^2)$ gives
\[ \lVert(d\phi)_x\rVert_{\alpha^*}^2 =\frac{x^{\mathsf T}(I-xx^{\mathsf T})x}{1-|x|^2} =|x|^2. \]
Integration of the exact one-form along curves proves \eqref{eq:distance-exact}.  Interchanging the endpoints gives \eqref{eq:symmetrization}.

The endpoint increment in \eqref{eq:distance-exact} telescopes, so the triangle inequality follows from the Klein triangle inequality.  If $x\ne y$, the Klein geodesic segment joining them is compactly contained in $\B^n$.  Equation \eqref{eq:dual-norm} is therefore bounded there by a constant smaller than one, and $\dF(x,y)>0$.  A sequence which is Cauchy for \eqref{eq:symmetrization} is Klein Cauchy.  Completeness of Klein hyperbolic space gives a limit in $\B^n$, and continuity of $\phi$ gives convergence for $\dF^{\mathrm s}$.  Separability follows from the Klein topology, while the density of $\mu_{n,b}$ is positive throughout $\B^n$.

Finally, $f(y)-f(x)\le\dF(x,y)$ is equivalent by \eqref{eq:distance-exact} to
\[ (f-\phi)(y)-(f-\phi)(x)\le\dK(x,y). \]
Since $\dK$ is symmetric, the last inequality is equivalent to the ordinary $1$-Lipschitz condition.  This proves \eqref{eq:affine-lip}.
	This completes the proof.
\end{proof}

The Klein distance is
\begin{equation}\label{eq:Klein-distance} \dK(x,y)=\arcosh\frac{1-\langle x,y\rangle}{\sqrt{(1-|x|^2)(1-|y|^2)}}, \end{equation}
and the two radial Funk distances are given in \eqref{eq:origin-distances}.

\subsection{The radial representation and master formula}

\begin{theorem}[Gamma representation and master distance formula]\label{thm:gamma-master}
\begin{equation}\label{eq:gamma-realization} \mathbf X_n\quad\text{has distribution }\mu_{n,\beta_n}. \end{equation}
Moreover,
\begin{align}
 \sinh^2\mathbf R_n&=\frac{U_n}{V_n},
 &\mathbf R_n&=\arsinh\sqrt{\frac{U_n}{V_n}},
 \label{eq:rho-gamma}\\
 \phi(\mathbf X_n)&=\log\cosh\mathbf R_n
 =\frac12\log\left(1+\frac{U_n}{V_n}\right).
 \label{eq:phi-gamma}
\end{align}
For $x=\tanh\rho\,\theta$ and $y=\tanh\sigma\,\eta$,
\begin{equation}\label{eq:polar-Klein} \cosh\dK(x,y) =\cosh(\rho-\sigma) +\frac12\sinh\rho\sinh\sigma\,|\theta-\eta|^2, \end{equation}
and
\begin{equation}\label{eq:master-Funk}
 \begin{split}
\dF(x,y) &=\arcosh\left( \cosh(\rho-\sigma) +\frac12\sinh\rho\sinh\sigma|\theta-\eta|^2 \right)\\
 &\quad+\log\cosh\sigma-\log\cosh\rho.
 \end{split}
\end{equation}
\end{theorem}

\begin{proof}
The variable $U_n/(U_n+V_n)$ has the beta distribution with parameters $n/2$ and $\beta_n$.  Polar integration in \eqref{eq:beta-measure} gives the same distribution for $|X|^2$ under $\mu_{n,\beta_n}$, independently of the uniform angular variable.  This proves \eqref{eq:gamma-realization}.  The identities in \eqref{eq:rho-gamma} and \eqref{eq:phi-gamma} follow from $|\mathbf X_n|=\tanh\mathbf R_n$.  Formula \eqref{eq:polar-Klein} is obtained by substituting $x=\tanh\rho\,\theta$ and $y=\tanh\sigma\,\eta$ into \eqref{eq:Klein-distance}.  Adding the endpoint increment in \eqref{eq:distance-exact} proves \eqref{eq:master-Funk}.
This completes the proof.
\end{proof}

\subsection{The three limit pyramids}

\begin{definition}[Directed Gaussian pyramid]\label{def:directed-gaussian}
For $q\in[0,1]$ and $k\ge0$, let
\[ G_q^k\coloneqq\bigl(\R\times\R^k,D_q,\gamma_{1/2}\otimes\gamma^k\bigr), \]
where $\gamma_{1/2}=N(0,1/2)$, $\gamma^k=N(0,I_k)$, and
\begin{equation}\label{eq:Dq} D_q((u,z),(v,z'))\coloneqq\sqrt{(u-v)^2+|z-z'|^2}+q(v-u). \end{equation}
Define
\begin{equation}\label{eq:Gq-pyramid} \mathcal G_q\coloneqq\overline{\bigcup_{k\ge0}\Pyr(G_q^k)}^{\,\Box}. \end{equation}
\end{definition}

\begin{definition}[Gaussian--chi Funk horocone pyramid]\label{def:horocone}
For $\beta>0$ and $k\ge0$, let
\[ H_\beta^k\coloneqq\bigl(\R^k\times(0,+\infty),D^{\mathrm{hor}},\gamma^k\otimes\chi_{2\beta}\bigr), \]
where $\chi_{2\beta}$ has density
\[ \frac{2^{1-\beta}}{\Gamma(\beta)}y^{2\beta-1}e^{-y^2/2}\,dy \qquad(y>0) \]
and
\begin{equation}\label{eq:horocone-distance} D^{\mathrm{hor}}((z,y),(z',y'))\coloneqq\arcosh\frac{|z-z'|^2+y^2+y'^2}{2yy'}+\log\frac{y}{y'}. \end{equation}
Define
\begin{equation}\label{eq:Hbeta-pyramid} \mathcal H_\beta\coloneqq\overline{\bigcup_{k\ge0}\Pyr(H_\beta^k)}^{\,\Box}. \end{equation}
\end{definition}

\begin{theorem}[Gaussian and horocone families are pyramids]\label{thm:gaussian-horocone-pyramids}
For every $q\in[0,1]$ and $\beta>0$, the families $(G_q^k)_{k\ge0}$ and $(H_\beta^k)_{k\ge0}$ are directed under exact domination.  Consequently, $\mathcal G_q$ and $\mathcal H_\beta$ defined by \eqref{eq:Gq-pyramid} and \eqref{eq:Hbeta-pyramid} are pyramids.
\end{theorem}

\begin{proof}
The coordinate projections preserve the corresponding product measures and do not increase the directed distances.  Thus $G_q^k\preceq G_q^{k+1}$ and $H_\beta^k\preceq H_\beta^{k+1}$.  For any $k,k'\ge0$, the generators with index $\max\{k,k'\}$ dominate both generators with indices $k$ and $k'$.  The directed-union closure principle in \Cref{thm:directed-union-closure} now shows that \eqref{eq:Gq-pyramid} and \eqref{eq:Hbeta-pyramid} are pyramids.
This completes the proof.
\end{proof}

\begin{definition}[Critical Funk star pyramid]\label{def:critical-star}
Let $m\ge1$ and let $p_i>0$ with $\sum_{i=1}^mp_i=1$.  On $[0,+\infty)\times\{1,\ldots,m\}$ identify all points $(0,i)$ and put
\begin{equation}\label{eq:critical-star-distance} d_{\rhd}((s,i),(t,j))\coloneqq\begin{cases}2(t-s)_+,&i=j,\\2t,&i\ne j.\end{cases} \end{equation}
Equip this space with the measure $e^{-s}\,ds\otimes\sum_{i=1}^mp_i\delta_i$ and denote the resulting qm-space by $T_m^{\rhd}(p_1,\ldots,p_m)$.  Define
\begin{equation}\label{eq:critical-star-pyramid} \mathcal S_{\rhd}\coloneqq\overline{\bigcup_{m,p_1,\ldots,p_m}\Pyr(T_m^{\rhd}(p_1,\ldots,p_m))}^{\,\Box}, \end{equation}
where the union is over all such $m$ and weight vectors.
\end{definition}

\subsection{The complete phase diagram}

\begin{theorem}[Complete natural phase diagram]\label{thm:main}
The Funk beta balls have the following natural-scale limits.
\begin{enumerate}[label=\textup{(\roman*)}]
\item Suppose that $\beta_n\to+\infty$, put
\begin{equation}\label{eq:qn} q_n\coloneqq\sqrt{\frac{n}{n+2\beta_n}}, \end{equation}
and assume that $q_n\to q\in[0,1]$.  Then
\begin{equation}\label{eq:gaussian-limit} \Pyr(\sqrt{2\beta_n}\,\BetaBall)\longrightarrow\mathcal G_q. \end{equation}
\item If $\beta_n\to\beta\in(0,+\infty)$, then
\begin{equation}\label{eq:fixed-beta-limit} \Pyr(\BetaBall)\longrightarrow\mathcal H_\beta. \end{equation}
\item If $\beta_n\to0$, then, without any assumption on $\beta_n\log n$,
\begin{equation}\label{eq:small-beta-limit} \Pyr(2\beta_n\BetaBall)\longrightarrow\mathcal S_{\rhd}. \end{equation}
\end{enumerate}
\end{theorem}

\section{General transfer of radial laws and endpoint potentials}\label{sec:general-transfer}

When the height law varies, convergence of the Gaussian base alone does not control the upper inclusion for the horocone,
while passage to the symmetric product loses the signed increment of the endpoint potential.  The following results transfer
both pieces of information to the pyramid limit through bounded measurements on compact intervals.

\subsection{Arbitrary radial laws and Funk horocones}

Write the finite-dimensional standard Gaussian base as
\begin{equation}\label{eq:Gaussian-base-space} \mathbb R_\gamma^k\coloneqq(\R^k,|\cdot|,\gamma^k). \end{equation}
For a Borel probability measure $\eta$ on $(0,+\infty)$, put
\begin{equation}\label{eq:general-height-generator}
 H_\eta^k\coloneqq\bigl(\R^k\times\supp\eta,D^{\mathrm{hor}},\gamma^k\otimes\eta\bigr)
\end{equation}
and define
\begin{equation}\label{eq:general-horocone-pyramid}
 \mathcal H_\eta\coloneqq\overline{\bigcup_{k\ge0}\Pyr(H_\eta^k)}^{\,\Box}.
\end{equation}
The distance in \Cref{eq:horocone-distance} does not depend on the height measure, so
\begin{equation}\label{eq:Hbeta-height-specialization} \mathcal H_\beta=\mathcal H_{\chi_{2\beta}}. \end{equation}
Indeed, the finite-dimensional generators on the two sides agree.

For an ordinary mm-space $Z$ and a Borel probability measure $\eta$ on $(0,+\infty)$, write
$\mathsf H(Z,\eta)$ for the Funk horocone defined in \Cref{eq:general-horocone-space}.  For an ordinary pyramid
$\mathcal P$, define
\begin{equation}\label{eq:joint-horocone-pyramid}
 \mathcal H(\mathcal P,\eta)
 \coloneqq\overline{\bigcup_{Z\in\mathcal P}
 \Pyr(\mathsf H(Z,\eta))}^{\,\Box}.
\end{equation}
This closure is an asymmetric pyramid.  Each principal pyramid in the union is nonempty and downward closed.  If
$Z_1,Z_2\in\mathcal P$, choose $W\in\mathcal P$ with $Z_1,Z_2\preceq W$.  Functoriality in
\Cref{thm:general-horocone} gives
$\mathsf H(Z_1,\eta),\mathsf H(Z_2,\eta)\preceq\mathsf H(W,\eta)$.  Thus the union before closure is directed, and
\Cref{thm:directed-union-closure} applies.

Let $\mu_n$ be a rotationally invariant probability measure on $\B^n$.  Write a random variable with law $\mu_n$ as
$X_n=R_n\Theta_n$, where $R_n\in[0,1)$ and $\Theta_n\sim\sigma_{n-1}$ are independent.  Associate with this radial law the height
\begin{equation}\label{eq:general-radial-height} \mathsf Y_n\coloneqq\sqrt{n(1-R_n^2)}. \end{equation}

\begin{theorem}[Horocone limit for general radial laws]\label{thm:radial-law-transfer}
Let $\eta$ be a probability measure with full support on $(0,+\infty)$, and suppose that
\begin{equation}\label{eq:height-convergence-assumption}
 (\mathsf Y_n)_*\mathbb P\Rightarrow\eta\qquad(n\to\infty).
\end{equation}
Then
\begin{equation}\label{eq:radial-law-pyramid-limit}
 \Pyr(\B^n,\dF,\mu_n)\longrightarrow\mathcal H_\eta\qquad(n\to\infty)
\end{equation}
weakly as asymmetric pyramids.
\end{theorem}

\begin{theorem}[Continuity of horocone-pyramid formation]\label{thm:height-law-pyramid-continuity}
Let $\eta_j$ and $\eta$ be probability measures with full support on $(0,+\infty)$, and suppose that
$\eta_j\Rightarrow\eta$ as $j\to\infty$.  Then
\begin{equation}\label{eq:height-law-pyramid-continuity}
 \mathcal H_{\eta_j}\longrightarrow\mathcal H_\eta\qquad(j\to\infty)
\end{equation}
weakly as asymmetric pyramids.
\end{theorem}

\subsection{Marked Gaussian products and endpoint potentials}

For an ordinary mm-space $Z=(Z,d_Z,\nu)$, a probability measure $\vartheta$ on $\R$, and $q\in[0,1]$, put
\begin{equation}\label{eq:marked-directed-space}
 \mathsf G_q(Z,\vartheta)\coloneqq(\R\times Z,D_{q,Z},\vartheta\otimes\nu),
\end{equation}
where
\begin{equation}\label{eq:Dq-general}
 D_{q,Z}((u,z),(v,z'))\coloneqq\sqrt{(u-v)^2+d_Z(z,z')^2}+q(v-u).
\end{equation}
For an ordinary pyramid $\mathcal P$, define the marked-product pyramid by
\begin{equation}\label{eq:marked-pyramid-operation}
 \mathsf G_q(\mathcal P,\vartheta)
 \coloneqq\overline{\bigcup_{Z\in\mathcal P}\Pyr(\mathsf G_q(Z,\vartheta))}^{\,\Box}.
\end{equation}
Since $\mathcal G_0$ is also the ordinary Gaussian pyramid, write
\begin{equation}\label{eq:G-q-mark-law}
 \mathcal G_{q,\vartheta}\coloneqq\mathsf G_q(\mathcal G_0,\vartheta).
\end{equation}
In particular,
\begin{equation}\label{eq:Gq-mark-specialization} \mathcal G_q=\mathcal G_{q,\gamma_{1/2}}. \end{equation}
Indeed, their finite-dimensional generators and Box closures agree.

\begin{proposition}[Transport of endpoint potentials]\label{prop:endpoint-potential-transport}
Suppose that ordinary pyramids converge weakly as $\mathcal P_n\to\mathcal P$, probability measures on $\R$ converge weakly
as $\vartheta_n\Rightarrow\vartheta$, and $q_n\to q\in[0,1]$.  Then
\begin{equation}\label{eq:directed-marked-product-limit}
 \mathsf G_{q_n}(\mathcal P_n,\vartheta_n)
 \longrightarrow\mathsf G_q(\mathcal P,\vartheta)\qquad(n\to\infty)
\end{equation}
weakly as asymmetric pyramids.
\end{proposition}

\subsection{Tangent limits of high horocones}

\begin{theorem}[High-horocone tangent theorem]\label{thm:horocone-tangent}
Let $h_j\to\infty$, let $\eta_j$ be a probability measure on $(0,+\infty)$, take $Y_j\sim\eta_j$, and put
$\zeta_j\coloneqq h_j-Y_j$.  Suppose that, for a probability measure $\vartheta$ on $\R$,
\begin{equation}\label{eq:tangent-height-assumptions}
 (\zeta_j)_*\mathbb P\Rightarrow\vartheta\qquad(j\to\infty),
 \qquad
 \frac{\zeta_j}{h_j}\longrightarrow0\qquad(j\to\infty)
 \quad\text{in probability}.
\end{equation}
Then
\begin{equation}\label{eq:horocone-tangent-limit}
 h_j\mathcal H_{\eta_j}\longrightarrow\mathcal G_{1,\vartheta}\qquad(j\to\infty)
\end{equation}
weakly as asymmetric pyramids.
\end{theorem}

When the base pyramid and the height law vary simultaneously, continuity for a fixed Gaussian base does not suffice to
pass directly to the profile limit.  The following joint continuity theorem resolves this obstruction.

\begin{theorem}[Joint continuity of horocone-pyramid formation]
\label{thm:joint-horocone-pyramid-continuity}
Suppose that ordinary pyramids converge weakly as $\mathcal P_n\to\mathcal P$ and arbitrary Borel probability measures on
$(0,+\infty)$ converge weakly as $\eta_n\Rightarrow\eta$, both as $n\to\infty$.  Then
\begin{equation}\label{eq:joint-horocone-pyramid-continuity}
 \mathcal H(\mathcal P_n,\eta_n)
 \longrightarrow\mathcal H(\mathcal P,\eta)
 \qquad(n\to\infty)
\end{equation}
weakly as asymmetric pyramids.
\end{theorem}

Proofs of
\Cref{thm:radial-law-transfer,thm:height-law-pyramid-continuity,thm:joint-horocone-pyramid-continuity}
are given in \Cref{app:general-transfer-proofs}.  Proofs of
\Cref{prop:endpoint-potential-transport,thm:horocone-tangent} are given in
\Cref{app:marked-product-proofs}.

\section{The large-beta directed Gaussian phase}\label{sec:large-beta}

The large-beta clause of \Cref{thm:main} retains Gaussian radial fluctuations as a directed mark on a transformed angular sphere.  A truncated-kernel comparison and marked-Gaussian absorption then identify the limit as $\mathcal G_q$.

Assume throughout this section that $\beta_n\to+\infty$.  Put
\begin{equation}\label{eq:large-beta-reference} c_n\coloneqq\sqrt{2\beta_n}, \qquad \bar\rho_n\coloneqq\arsinh\sqrt{\frac{n}{2\beta_n}}. \end{equation}
Then $q_n=\tanh\bar\rho_n$ by \eqref{eq:qn}.

\begin{theorem}[Universal radial central limit theorem]\label{thm:radial-clt}
For the radius in \eqref{eq:rho-gamma},
\begin{equation}\label{eq:radial-clt} c_n(\mathbf R_n-\bar\rho_n)\overset{\mathrm{d}}{\longrightarrow} N(0,1/2). \end{equation}
The limiting radial variable is independent of every fixed finite collection of the angular variables $\sqrt n\,\Theta_{n,j}$.
\end{theorem}

\begin{proof}
\Cref{lem:radial-linearization} replaces the radial fluctuation by the linear combination in \eqref{eq:radial-linear-part}, with an error tending to zero in probability uniformly over the ratio $\beta_n/n$.  The two logarithmic Gamma fluctuations converge jointly to independent standard Gaussian variables, and the squares of their deterministic coefficients sum to $1/2$.  Every subsequential coefficient limit therefore gives $N(0,1/2)$, which proves \eqref{eq:radial-clt}.  The angular variable is independent of $(U_n,V_n)$ for every $n$, and its first finitely many coordinates converge jointly to independent standard Gaussian variables.  This completes the proof.
\end{proof}

For $a_n\to+\infty$, equip $\sqrt nS^{n-1}$ with
\begin{equation}\label{eq:transformed-sphere-metric} \delta_n(z,z')\coloneqq2a_n\arsinh\frac{|z-z'|}{2a_n}. \end{equation}
\Cref{thm:transformed-sphere-gaussian} proves that these transformed angular spheres, as well as the chord spheres, converge to the Gaussian pyramid.  The projection and radial-lifting estimates are collected in \Cref{app:profile-estimates}.

\begin{corollary}[Marked Gaussian absorption]\label{thm:marked-gaussian-absorption}
Let $a_n\to+\infty$, let probability measures $\nu_n\longrightarrow\gamma_{1/2}$ weakly on $\R$, and let $q_n\to q\in[0,1]$.  On $\R\times\sqrt nS^{n-1}$ put the product measure $\nu_n\otimes\sigma_{n-1}$ and the directed distance
\begin{equation}\label{eq:marked-comparison-distance} D_n((u,z),(v,z')) \coloneqq\sqrt{(u-v)^2+\delta_n(z,z')^2}+q_n(v-u), \end{equation}
where $\delta_n$ is defined by \eqref{eq:transformed-sphere-metric}.  Then
\[ \Pyr(\R\times\sqrt nS^{n-1},D_n, \nu_n\otimes\sigma_{n-1}) \longrightarrow\mathcal G_q. \]
\end{corollary}

\begin{proof}
Put $Z_n\coloneqq(\sqrt nS^{n-1},\delta_n,\sigma_{n-1})$ and
$\mathcal P_n\coloneqq\Pyr(Z_n)$.  By \Cref{thm:transformed-sphere-gaussian},
$\mathcal P_n\to\mathcal G_0$.  For every $W\preceq Z_n$, the product of its domination map with the identity on $\R$
realizes
$\mathsf G_{q_n}(W,\nu_n)\preceq\mathsf G_{q_n}(Z_n,\nu_n)$.  Conversely, $Z_n\in\mathcal P_n$, so
\[
 \mathsf G_{q_n}(\mathcal P_n,\nu_n)
 =\Pyr\bigl(\mathsf G_{q_n}(Z_n,\nu_n)\bigr)
 =\Pyr(\R\times\sqrt nS^{n-1},D_n,\nu_n\otimes\sigma_{n-1}).
\]
Apply \Cref{prop:endpoint-potential-transport} to
$\mathcal P_n\to\mathcal G_0$, $\nu_n\Rightarrow\gamma_{1/2}$, and $q_n\to q$.  The right-hand side converges to
$\mathsf G_q(\mathcal G_0,\gamma_{1/2})$, which equals $\mathcal G_q$ by
\Cref{eq:Gq-mark-specialization}.  This completes the proof.
\end{proof}

\Cref{lem:large-beta-kernel-comparison} gives measure-preserving comparison maps on radial windows whose omitted mass and truncated-kernel error both tend to zero.  Therefore, \Cref{lem:bounded-kernel-transfer} gives the same subsequential pyramid limits for the Funk balls and comparison spaces.  \Cref{thm:radial-clt,thm:marked-gaussian-absorption} identify their limit as $\mathcal G_q$.
This proves the large-beta clause of \Cref{thm:main}.

\section{The fixed-beta Gaussian--chi Funk horocone phase}\label{sec:fixed-beta}

For the fixed-beta clause of \Cref{thm:main}, the entire chi-distributed height remains in the limit as $\beta_n$ approaches a positive constant.  Horospherical comparison produces a Funk horocone, and the upper inclusion compactifies radial profiles with a path metric before evaluating them over Gaussian-pyramid bases.

\Cref{thm:general-horocone} implies that a horocone with height law $\chi_{2\beta}$ over a Gaussian-pyramid member belongs to $\mathcal H_\beta$, and that joint box convergence of the base and height law passes to the horocones.  This supplies the continuity and closure inputs for the two inclusions in \Cref{thm:fixed-beta-convergence}.

\subsection{Horospherical comparison}

For $x=\tanh\rho\,\theta$, define
\begin{equation}\label{eq:horospherical-map} z\coloneqq\sqrt n\,\theta, \qquad y\coloneqq\frac{\sqrt n}{\cosh\rho} =\sqrt{n(1-|x|^2)}. \end{equation}
For $s,t>0$ and $r\ge0$, put
\begin{equation}\label{eq:upper-half-Funk} \Omega_{s,t}(r) \coloneqq\arcosh\frac{r^2+s^2+t^2}{2st}+\log\frac{s}{t}. \end{equation}

\begin{theorem}\label{thm:horospherical-transfer}
There is a function $\eps(R)\downarrow0$ such that, whenever $\rho,\sigma\ge R$,
\begin{equation}\label{eq:horospherical-error} \left|\dF((\rho,\theta),(\sigma,\eta)) -\Omega_{\sqrt n/\cosh\rho,\sqrt n/\cosh\sigma} (\sqrt n|\theta-\eta|)\right| \le\eps(R). \end{equation}
Let
\begin{equation}\label{eq:height-law} Y_n^2\coloneqq\frac{nV_n}{U_n+V_n} \end{equation}
and define
\begin{equation}\label{eq:upper-half-comparison-space} \widehat X_n \coloneqq\left(\sqrt nS^{n-1}\times(0,+\infty),D^{\mathrm{hor}}, \sigma_{n-1}\otimes (Y_n)_*\mathbb P\right). \end{equation}
If $\beta_n$ is bounded above and $b_n>0$ is bounded above, then
\begin{equation}\label{eq:horospherical-box-transfer} \Box(b_n\BetaBall,b_n\widehat X_n) \longrightarrow0. \end{equation}
\end{theorem}

\begin{proof}
Applying \Cref{lem:general-horospherical-approximation} with $L=R$ proves \eqref{eq:horospherical-error} with $\eps(R)\coloneqq\delta(R)$.

The Gamma representation gives $U_n/n\to1/2$ in probability as $n\to\infty$.  Boundedness of $\mathbb EV_n=\beta_n$ and Markov's inequality give $V_n/n\to0$ in probability as $n\to\infty$.  Therefore, \eqref{eq:rho-gamma} implies that $\mathbf R_n\to+\infty$ in probability.  Choose $R_n\to+\infty$ such that $\mathbb P(\mathbf R_n\ge R_n)\to1$ as $n\to\infty$.  Radial--angular independence and \eqref{eq:height-law} show that the map in \eqref{eq:horospherical-map} is measure preserving from the Funk ball to $\widehat X_n$.  Applying \Cref{lem:general-horospherical-approximation} with $L=R_n$ bounds its distortion on $\{\rho\ge R_n\}$ after scaling by $(\sup_nb_n)\delta(R_n)$.  The exceptional mass and this distortion both tend to zero as $n\to\infty$.  The relation characterization of box distance now gives \eqref{eq:horospherical-box-transfer}.  This completes the proof.
\end{proof}

If $\beta_n\to\beta\in(0,+\infty)$, then \eqref{eq:height-law} gives
\begin{equation}\label{eq:height-fixed-beta} Y_n\overset{\mathrm{d}}{\longrightarrow}\sqrt{2V}, \qquad V\sim\Gamma(\beta,1), \end{equation}
which is the chi distribution $\chi_{2\beta}$.

\begin{theorem}[Fixed-beta lower and upper inclusions]\label{thm:fixed-beta-convergence}
Suppose that $\beta_n\to\beta\in(0,+\infty)$.  Every $H_\beta^k$ belongs to each subsequential pyramid limit of $\widehat X_n$.  Conversely, every weak limit of bounded one-sided measurements of $\widehat X_n$ belongs to $\mathcal M(\mathcal H_\beta;N,R)$ for the corresponding $N$ and $R$.  Consequently,
\[ \Pyr(\widehat X_n)\longrightarrow\mathcal H_\beta. \]
\end{theorem}

\begin{proof}
The Gamma representation, $U_n/n\to1/2$, and $V_n\Rightarrow V\sim\Gamma(\beta,1)$ give, by Slutsky's theorem,
$Y_n^2\Rightarrow2V$ in \eqref{eq:height-law}, and hence $Y_n\Rightarrow\chi_{2\beta}$.  This height law has full support
on $(0,+\infty)$.  Applying \Cref{thm:radial-law-transfer} to the radial laws of the Funk beta balls and using
\Cref{eq:Hbeta-height-specialization} therefore gives
\[
 \Pyr(\BetaBall)\longrightarrow\mathcal H_{\chi_{2\beta}}=\mathcal H_\beta.
\]
On the other hand, \Cref{thm:horospherical-transfer} with $b_n=1$ gives
$\Box(\BetaBall,\widehat X_n)\to0$.  Thus $\Pyr(\widehat X_n)$ also converges to $\mathcal H_\beta$, and the
two-inclusion characterization of pyramid convergence gives the lower and upper inclusions stated in the theorem.  This
completes the proof.
\end{proof}

The Funk-beta-ball convergence obtained in the proof of \Cref{thm:fixed-beta-convergence} is the fixed-beta clause of
\Cref{thm:main}.

\section{The small-beta critical star phase}\label{sec:small-beta}

Exact cancellation drives the small-beta clause of \Cref{thm:main}: cross-ray distances become independent of the source height in the limit, and the logarithmic height becomes an exponential variable.  Finite angular blocks give the lower inclusion, while spherical proximity and finite quantization give the upper inclusion.

\Cref{tab:poincare-funk-small-beta} isolates the directed effect of the endpoint potential from the symmetric beta-ball phase diagram \cite{yokota-poincare-beta-balls}.

The symbol $\operatorname{Exp}(1)$ denotes the exponential distribution of rate one, with density $e^{-t}$ on $[0,+\infty)$.
In the Poincar\'e column, $\tau$ denotes a finite limit of $\beta_n\log n$.

\begin{table}[ht]
\centering
\small
\setlength{\tabcolsep}{5pt}
\renewcommand{\arraystretch}{1.15}
\begin{tabular}{@{}p{0.25\textwidth}p{0.31\textwidth}p{0.31\textwidth}@{}}
\toprule
\raggedright Feature & \raggedright Poincar\'e (finite product limit) & \raggedright Forward Funk \tabularnewline
\midrule
Geometry & symmetric & directed \\
Small-beta scale & $\beta_n$ & $2\beta_n$ \\
Radial mark & $\tau+\operatorname{Exp}(1)$ & $\operatorname{Exp}(1)$ \\
Cross-branch distance & $s+t-\tau$ & $2t$ \\
$\beta_n\log n$ & survives as $\tau$ & cancels exactly \\
\bottomrule
\end{tabular}
\caption{The small-beta distinction between the symmetric and forward metrics.}
\label{tab:poincare-funk-small-beta}
\end{table}

Define
\begin{equation}\label{eq:small-beta-radial} S_n\coloneqq-2\beta_n\log Y_n =-\beta_n\log\frac{nV_n}{U_n+V_n}. \end{equation}

\begin{theorem}[Small-shape Gamma and exponential radial limits]\label{thm:small-beta-radial}
Suppose that $\beta_n\to0$.  Then
\begin{align}
\beta_n\log^+\frac{V_n}{n}&\longrightarrow0, \label{eq:small-beta-upper-gamma}\\
\beta_n\log\frac{U_n+V_n}{n}&\longrightarrow0 \label{eq:small-beta-gamma-sum}
\end{align}
in probability, and
\begin{equation}\label{eq:Sn-exp} S_n\overset{\mathrm{d}}{\longrightarrow}\Exp(1). \end{equation}
\end{theorem}

\begin{proof}
For every $\eps>0$, Markov's inequality gives
\[ \mathbb P\left(\beta_n\log^+\frac{V_n}{n}>\eps\right) \le\mathbb P\left(V_n>ne^{\eps/\beta_n}\right) \le\frac{\beta_n}{n}e^{-\eps/\beta_n}, \]
which tends to zero.  Since $U_n/n\to1/2$ in probability,
$\beta_n\log(U_n/n)\to0$.  On the event $U_n/n\ge1/4$,
\[ 0\le\log\left(1+\frac{V_n}{U_n}\right) \le\log5+\log^+\frac{V_n}{n}. \]
Combining this bound with
\[ \log\frac{U_n+V_n}{n} =\log\frac{U_n}{n} +\log\left(1+\frac{V_n}{U_n}\right) \]
proves \eqref{eq:small-beta-gamma-sum}.

For every $t\in\R$,
\[ \mathbb E\exp(-it\beta_n\log V_n) =\frac{\Gamma((1-it)\beta_n)}{\Gamma(\beta_n)} \longrightarrow\frac1{1-it}, \]
because $z\Gamma(z)\to1$ as $z\to0$ with positive real part.  The limiting function is the characteristic function of $\Exp(1)$.  L\'evy's continuity theorem gives $-\beta_n\log V_n\overset{\mathrm{d}}{\longrightarrow}\Exp(1)$, and \eqref{eq:small-beta-gamma-sum} transfers this limit to $S_n$.  This completes the proof.
\end{proof}

\begin{theorem}[Same-ray and cross-ray kernels]\label{thm:small-beta-kernels}
Let $I\Subset(0,+\infty)$ be compact.  For $s,t\in I$, set
$y=e^{-s/(2\beta_n)}$ and $y'=e^{-t/(2\beta_n)}$.  Then
\begin{equation}\label{eq:small-same-ray} 2\beta_n\Omega_{y,y'}(0)=2(t-s)_+. \end{equation}
For every $0<c<C<+\infty$, uniformly for $s,t\in I$ and $c\le r\le C$,
\begin{equation}\label{eq:small-cross} 2\beta_n\Omega_{y,y'}(r)\longrightarrow2t. \end{equation}
\end{theorem}

\begin{proof}
At $r=0$, the symmetric term is $|\log(y/y')|$, and therefore
\[ 2\beta_n\Omega_{y,y'}(0)=|t-s|+(t-s)=2(t-s)_+. \]
This proves \eqref{eq:small-same-ray}.  For the cross-ray assertion, the lower bounds on $s,t$, and $r$ make
\[ \frac{r^2+y^2+y'^2}{2yy'}\longrightarrow+\infty \]
uniformly.  More explicitly,
\[ \frac{r^2+y^2+y'^2}{2yy'} \ge\frac{c^2}{2}\exp\left(\frac{s+t}{2\beta_n}\right) \ge\frac{c^2}{2}\exp\left(\frac{\min I}{\beta_n}\right), \qquad c^2\le r^2+y^2+y'^2\le C^2+2. \]
The first bound makes the inverse-hyperbolic-cosine remainder below vanish uniformly, while the second pair makes $\log(r^2+y^2+y'^2)$ uniformly bounded.  The identity
\[ \arcosh w=\log(2w)+ \log\frac{1+\sqrt{1-w^{-2}}}{2} \]
shows that the difference between the inverse hyperbolic cosine and the logarithm tends uniformly to zero.  Moreover,
\[ 2\beta_n\log\frac{r^2+y^2+y'^2}{yy'} =s+t+2\beta_n\log(r^2+y^2+y'^2) \longrightarrow s+t \]
uniformly on the stated set.  Adding the endpoint increment
$2\beta_n\log(y/y')=t-s$ proves \eqref{eq:small-cross}.
	This completes the proof.
\end{proof}

On a three-branch generator, let $O$ be the common vertex and take $x=(s,1)$ and $y=(t,3)$ with $0<s<t$.  \Cref{fig:critical-directed-star} displays their cross-branch asymmetry and the same-branch kernel in a separate inset.

\begin{figure}[H]
\centering
\begin{tikzpicture}[
  x=0.95cm,
  y=0.95cm,
  >=Latex,
  ray/.style={draw=black!65, line width=0.9pt, -{Latex[length=2.2mm]}},
  directed/.style={draw=black, line width=1pt, -{Latex[length=2mm]}},
  reverse/.style={draw=black!55, line width=0.9pt, -{Latex[length=2mm]}},
  marked point/.style={circle, fill=black, inner sep=1.6pt}
]
  \coordinate (O) at (0,0);
  \coordinate (E1) at (4.2,2.1);
  \coordinate (E2) at (4.6,0);
  \coordinate (E3) at (4.2,-2.1);
  \coordinate (X) at (1.8,0.9);
  \coordinate (Y) at (2.8,-1.4);

  \draw[ray] (O) -- (E1);
  \draw[ray] (O) -- (E2);
  \draw[ray] (O) -- (E3);

  \node[marked point, label={[xshift=-2pt,yshift=2pt]above left:$O$}] at (O) {};
  \node[marked point, label=above left:{$x=(s,1)$}] at (X) {};
  \node[marked point, label=below left:{$y=(t,3)$}] at (Y) {};

  \node[anchor=west] at (E1) {$1\ (p_1)$};
  \node[anchor=west] at (E2) {$2\ (p_2)$};
  \node[anchor=west] at (E3) {$3\ (p_3)$};

  \draw[directed] (X) .. controls (4.0,0.7) and (4.0,-0.7) ..
    node[pos=0.65, left, fill=white, inner sep=1.5pt] {$2t$} (Y);
  \draw[reverse] (Y) .. controls (4.7,-0.7) and (4.7,0.7) ..
    node[pos=0.65, right, fill=white, inner sep=1.5pt] {$2s$} (X);

  \coordinate (L) at (6.3,0);
  \coordinate (E) at (11.3,0);
  \coordinate (A) at (7.5,0);
  \coordinate (B) at (9.7,0);
  \draw[ray] (L) -- (E);
  \node[marked point, label=below:{$(s,i)$}] at (A) {};
  \node[marked point, label=below:{$(t,i)$}] at (B) {};
  \draw[directed] ([yshift=5pt]A) -- node[midway, above, fill=white, inner sep=1pt]
    {$2(t-s)$} ([yshift=5pt]B);
  \draw[reverse] ([yshift=-5pt]B) -- node[midway, below, fill=white, inner sep=1pt]
    {$0$} ([yshift=-5pt]A);
  \path (A) -- node[midway, above=14pt] {$0<s<t$} (B);
\end{tikzpicture}
\caption{A three-branch generator $T_3^{\rhd}(p_1,p_2,p_3)$ of the critical star pyramid.  Branch $i$ carries the radial measure $p_i e^{-s}\,ds$.  The left diagram shows that the cross-branch distances from $x$ to $y$ and from $y$ to $x$ are $2t$ and $2s$, respectively.  For $0<s<t$, the right inset shows distances $2(t-s)$ outward and $0$ inward on one branch.}
\label{fig:critical-directed-star}
\end{figure}
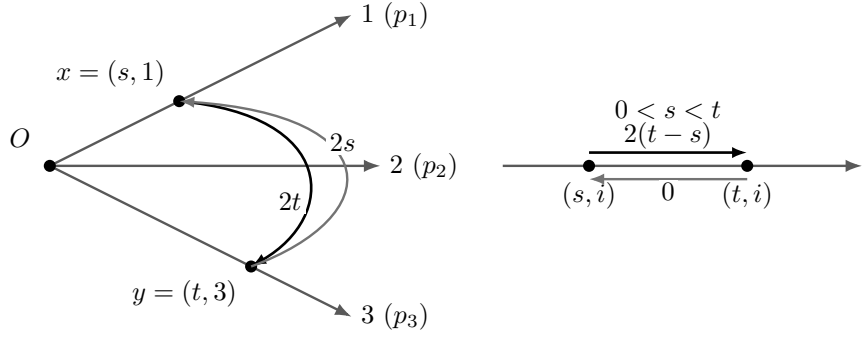

\Cref{thm:spherical-blocks-nearby} supplies the two angular inputs used below: separated blocks with prescribed limiting masses and bounded scaled distance between two sets of uniformly positive measure.

\begin{theorem}\label{thm:finite-star-lower}
For every $m\ge1$ and every positive weight vector $(p_1,\ldots,p_m)$, the critical star $T_m^{\rhd}(p_1,\ldots,p_m)$ belongs to every subsequential pyramid limit of $2\beta_n\widehat X_n$.
\end{theorem}

\begin{proof}
Fix a compact interval $I=[\delta,M]\Subset(0,+\infty)$ and a positive mass accuracy.  Use \Cref{thm:spherical-blocks-nearby}(a) to choose angular blocks with the prescribed approximate masses and a fixed positive lower bound for their scaled chord separation.  On
\[ A_n\coloneqq\{(\sqrt n\theta,Y_n)\colon \theta\in\textstyle\bigcup_iB_{n,i},\ S_n\in I\}, \]
define
\[ P_n(\sqrt n\theta,Y_n)\coloneqq(S_n,i) \qquad(\theta\in B_{n,i}). \]
For two points in the same block, monotonicity of $r\mapsto\Omega_{y,y'}(r)$ and \eqref{eq:small-same-ray} give
\[ d_{\rhd}(P_n(\sqrt n\theta,Y_n),P_n(\sqrt n\eta,Y_n')) \le2\beta_nD^{\mathrm{hor}}((\sqrt n\theta,Y_n),(\sqrt n\eta,Y_n')). \]
For distinct blocks, evaluate the source distance at the fixed lower separation cutoff.  Monotonicity and \eqref{eq:small-cross} give the same inequality with an additive error tending to zero, uniformly for $S_n,S_n'\in I$.

The radial and angular variables are independent.  By \Cref{thm:small-beta-radial}, the restricted pushforward measures tend, as the block-mass accuracy tends to zero, to
\[ (e^{-s}\,ds|_I)\otimes\sum_{i=1}^mp_i\delta_i. \]
As the intervals exhaust $(0,+\infty)$ and the block-mass accuracy tends to zero, the omitted mass and additive error vanish while these restricted measures approach the full star measure.  The diagonal selection in \Cref{lem:diagonal-high-mass-selection}, followed by \Cref{lem:high-mass-domination}, proves the assertion.
	This completes the proof.
\end{proof}

\Cref{thm:finite-profile-quantization} collects the compactness, cross-profile comparison, and finite quantization estimates needed for the upper inclusion.

\begin{theorem}[Finite-star family is a pyramid]\label{thm:finite-star-pyramid}
The set $\mathcal S_{\rhd}$ in \eqref{eq:critical-star-pyramid} is an asymmetric pyramid.
\end{theorem}

\begin{proof}
Each finite star is a qm-space.  Indeed, let
\[
 d_{\mathrm{tree}}((s,i),(t,j))
 \coloneqq\begin{cases}|s-t|,&i=j,\\s+t,&i\ne j.
 \end{cases}
\]
Then
\begin{equation}\label{eq:star-tree-potential} d_{\rhd}((s,i),(t,j)) =d_{\mathrm{tree}}((s,i),(t,j))+t-s. \end{equation}
The branch coordinate is $1$-Lipschitz for $d_{\mathrm{tree}}$, and the max symmetrization is
\begin{equation}\label{eq:star-symmetrization}
 d_{\rhd}^{\mathrm s}((s,i),(t,j))
 =\begin{cases}
 2|s-t|,&i=j,\\
 2\max\{s,t\},&i\ne j.
 \end{cases}
\end{equation}
It is complete and separable after the common-vertex quotient, and the star measure has full support.

The generators are directed under domination.  Given stars with weights $(p_i)_{i=1}^m$ and $(q_j)_{j=1}^\ell$, form the star with branch labels $(i,j)$ and weights $p_iq_j$.  The maps
\[ (s,(i,j))\longmapsto(s,i), \qquad (s,(i,j))\longmapsto(s,j) \]
are measure preserving.  If a retained label agrees while the full labels differ, the source distance is $2t$ and the target distance is $2(t-s)_+\le2t$.  In all other cases the target distance is either equal to or below the source distance.  The product-refinement star therefore dominates both generators.  Applying \Cref{thm:directed-union-closure} to this directed family proves that \eqref{eq:critical-star-pyramid} is a pyramid.  This completes the proof.
\end{proof}

The lower inclusion follows from \Cref{thm:finite-star-lower} and the generating definition \eqref{eq:critical-star-pyramid}.

For the upper inclusion, take bounded measurements whose pushforward measures converge to $\mu$ and fix a compact radial interval $I\Subset(0,+\infty)$.  Their radial profiles form a compact family of coordinatewise nondecreasing $2$-Lipschitz maps on $I$.  Pass to a weak limit $\varpi$ of the angular profile measures.  The same-ray kernel gives the vertical inequality on $\supp\varpi$, while positive-measure spherical proximity and the cross-ray kernel give the cross-profile inequality.  Radial--angular independence identifies the evaluation pushforward over $\varpi\otimes(e^{-s}\,ds|_I)$ as a submeasure of $\mu$.  \Cref{thm:finite-profile-quantization} replaces this compact profile family by a finite critical-star measurement.  Its error is the sum of the profile quantization error and twice the exponential mass outside $I$, so it vanishes as the quantization is refined and $I$ exhausts $(0,+\infty)$.  This proves the upper inclusion.

Since \Cref{thm:finite-star-pyramid} shows that the target is a pyramid, the bounded-measurement convergence criterion identifies the comparison-space limit.  The factor $2\beta_n$ is bounded, so \Cref{thm:horospherical-transfer} transfers this limit to the Funk balls.  Thus the finite-star lower inclusion, finite profile quantization, and the pyramid property of the finite-star family prove the small-beta clause of \Cref{thm:main}.

\begin{remark}[Cancellation and criticality]\label{rem:cancellation-criticality}
For two angular blocks at scaled chord distance of constant order, the two radial factors in the Klein cosh expression contribute two copies of the term $\beta_n\log n$ after multiplication by $2\beta_n$, while $|\theta-\eta|$ contributes their negative sum.  The cancellation is exact in the height $Y_n=\sqrt{n(1-|\mathbf X_n|^2)}$, leaving only $-\beta_n\log V_n$ in the limiting radial mark.  In every finite dimension, \eqref{eq:dual-norm} is strictly below one and distinct points have positive Funk distance.  Its supremum over the ball equals one, however, and the limiting Gaussian endpoint, horocones, and critical stars can have distinct points at zero distance in one orientation.  The one-sided vanishing is therefore a limit manifestation of the globally critical exact potential, rather than a finite-dimensional degeneracy.
\end{remark}

\section{Identification and separation}\label{sec:identification-separation}

Variance identifies the directed Gaussian parameter, while the positive tail exponent identifies the horocone parameter.  The same invariants distinguish the directed star from both families.

\begin{definition}[Maximal bounded-measurement variance]\label{def:variance-invariant}
For an asymmetric pyramid $\mathcal P$, define
\[ \mathbf{V}^{\to}(\mathcal P) \coloneqq\sup_{R>0}\sup_{\nu\in\mathcal M(\mathcal P;1,R)}\Var(\nu). \]
\end{definition}
The symmetric theory extending classical invariants such as observable variance from pm-spaces to pyramids is developed in \cite{esaki-kazukawa-mitsuishi2024invariants}.  Here we use the directed/asymmetric counterpart.

\begin{theorem}[Identification of the directed Gaussian family]\label{thm:Gq-classification}
For every $q\in[0,1]$,
\begin{equation}\label{eq:Gq-variance} \mathbf{V}^{\to}(\mathcal G_q)=1+q^2. \end{equation}
Consequently, $\mathcal G_q=\mathcal G_{q'}$ if and only if $q=q'$.  If $\beta_n\to+\infty$, then
\[ \Pyr(\sqrt{2\beta_n}\,\BetaBall) \]
converges if and only if $q_n$ converges.
\end{theorem}

\begin{proof}
By the exact-potential identity, every one-sided observable on $G_q^k$ has the form
\[ f(u,z)=qu+g(u,z) \]
with $g$ Euclidean $1$-Lipschitz.  Gaussian Poincar\'e and Rademacher's theorem give
\[ \Var(f) \le\mathbb E\left[\frac12(q+\partial_ug)^2+|\nabla_zg|^2\right]. \]
At almost every point, $(\partial_ug)^2+|\nabla_zg|^2\le1$, and hence
\[ \frac12(q+\partial_ug)^2+|\nabla_zg|^2 \le1+q^2-\frac12(\partial_ug-q)^2 \le1+q^2. \]
Exact factors and bounded measurement limits preserve this bound.

For the reverse inequality, take
\[ g(u,z)=qu+\sqrt{1-q^2}\,z_1. \]
It is Euclidean $1$-Lipschitz, and the corresponding one-sided observable
\[ f(u,z)=2qu+\sqrt{1-q^2}\,z_1 \]
has variance $1+q^2$.  Monotone clipping preserves the one-sided Lipschitz property, and the clipped variances tend to this value.  This proves \eqref{eq:Gq-variance} and distinguishes the parameters.

If $q_n$ converges, the large-beta clause of \Cref{thm:main} gives convergence.  Conversely, the same clause shows that each cluster point $q$ of $(q_n)$ gives the subsequential limit $\mathcal G_q$.  Convergence of the full pyramid sequence makes all these limits equal, and \eqref{eq:Gq-variance} makes all cluster points equal.  Compactness of $[0,1]$ proves convergence of $q_n$.  This completes the proof.
\end{proof}

\begin{definition}[Positive tail exponent]\label{def:tail-exponent}
For a probability measure $\nu$ on $\R$, let $\operatorname{Med}(\nu)$ be its closed median interval.  For an asymmetric pyramid $\mathcal P$ and $\lambda>0$, put
\begin{align}
 \mathsf E_\lambda(\mathcal P)
 &\coloneqq\sup_{R>0}\sup_{\nu\in\mathcal M(\mathcal P;1,R)}
 \inf_{m\in\operatorname{Med}(\nu)}
 \int_\R e^{\lambda(t-m)_+}\,d\nu(t),
 \label{eq:tail-functional}\\
 \Lambda^+(\mathcal P)
 &\coloneqq\sup\{\lambda>0\colon\mathsf E_\lambda(\mathcal P)<+\infty\}.
 \label{eq:tail-exponent}
\end{align}
\end{definition}

\begin{theorem}[Identification of the fixed-beta family]\label{thm:Hbeta-classification}
For every $\beta>0$,
\begin{equation}\label{eq:Hbeta-tail-exponent} \Lambda^+(\mathcal H_\beta)=\beta. \end{equation}
Consequently, $\mathcal H_\beta=\mathcal H_{\beta'}$ if and only if $\beta=\beta'$.  If $(\beta_n)$ stays in a compact subinterval of $(0,+\infty)$, then $\Pyr(\BetaBall)$ converges if and only if $\beta_n$ converges.
\end{theorem}

\begin{proof}
It is enough to prove a uniform upper bound on every finite generator.  Let $f$ be a bounded one-sided observable on $H_\beta^k$, and let $Z\sim\gamma^k$ and $Y\sim\chi_{2\beta}$ be independent.  Put
\[ \xi(y)\coloneqq-2\log y, \qquad h(z)\coloneqq f(z,1). \]
At height one the directed metric is symmetric and equals $2\arsinh(|z-z'|/2)$, so $h$ is Euclidean $1$-Lipschitz.  The vertical directed and symmetrized inequalities give
\begin{equation}\label{eq:vertical-tail-control} f(z,y)\le h(z)+\xi(y)_+, \qquad |f(z,y)-h(z)|\le|\xi(y)|. \end{equation}
Gaussian Poincar\'e gives $\Var(h(Z))\le1$.  Therefore
$\mathbb P(h(Z)\ge\mathbb Eh(Z)-2)\ge3/4$.  Choose $C_\beta$ with $\mathbb P(|\xi(Y)|\le C_\beta)>3/4$.  Independence and the second inequality in \eqref{eq:vertical-tail-control} show that every median $m_f$ of $f(Z,Y)$ satisfies
\[ m_f\ge\mathbb Eh(Z)-2-C_\beta. \]
The first inequality then gives
\[ (f-m_f)_+ \le(h-\mathbb Eh)_++\xi(Y)_++2+C_\beta. \]
For every $\lambda>0$, the Gaussian exponential estimate gives
\[ \mathbb E e^{\lambda(h-\mathbb Eh)_+} \le1+e^{\lambda^2/2}. \]
Writing $Y^2=2V$ with $V\sim\Gamma(\beta,1)$, we have
\[ \mathbb E e^{\lambda\xi(Y)_+}<+\infty \quad\Longleftrightarrow\quad\lambda<\beta. \]
Indeed, near zero the Gamma density is a constant multiple of $v^{\beta-1}$ and $e^{\lambda\xi(\sqrt{2v})}=(2v)^{-\lambda}$.  This proves $\mathsf E_\lambda(\mathcal H_\beta)<+\infty$ for $\lambda<\beta$.

For the reverse inequality, $\xi(y)=-2\log y$ is one-sided $1$-Lipschitz for the vertical horocone distance.  Its monotone clippings are bounded one-sided observables, and their centered upper exponential moments diverge with the clipping level whenever $\lambda\ge\beta$.  This proves \eqref{eq:Hbeta-tail-exponent}.

The fixed-beta clause of \Cref{thm:main} shows that each cluster point $\beta$ of a bounded positive parameter sequence gives the subsequential Funk-ball limit $\mathcal H_\beta$.  If the full sequence converges, all such targets coincide.  Equation \eqref{eq:Hbeta-tail-exponent} makes all cluster points equal, and compactness of the parameter interval proves convergence of $\beta_n$.  The converse is again the fixed-beta clause of \Cref{thm:main}; its comparison-space inclusions are supplied by \Cref{thm:fixed-beta-convergence}.  This completes the proof.
\end{proof}

\begin{proposition}[Pairwise separation of the phase pyramids]
\label{prop:phase-pyramid-separation}
For every $q\in[0,1]$ and $\beta>0$, the pyramids
$\mathcal G_q$, $\mathcal H_\beta$, and $\mathcal S_{\rhd}$ are pairwise
distinct.  More precisely,
\begin{align}
 \Lambda^+(\mathcal G_q)&=+\infty,
 &\Lambda^+(\mathcal S_{\rhd})&=\frac12,
 \label{eq:phase-separation-tail}\\
 \mathbf{V}^{\to}(\mathcal S_{\rhd})&=4,
 &\mathbf{V}^{\to}(\mathcal H_{1/2})&\ge\frac{\pi^2}{2}.
 \label{eq:phase-separation-variance}
\end{align}
\end{proposition}

\begin{proof}
Let $f$ be a bounded one-sided observable on $G_q^k$.  Applying the
one-sided inequality in both orders and using \eqref{eq:Dq} gives
\[ |f(u,z)-f(v,z')| \le\sqrt{(u-v)^2+|z-z'|^2}+q|u-v| \le2\sqrt{(u-v)^2+|z-z'|^2}. \]
After writing the radial Gaussian coordinate as $U=W/\sqrt2$, the observable
is therefore a $2$-Lipschitz function of a standard Gaussian vector.  The
Gaussian exponential estimate gives
\begin{equation}\label{eq:Gq-centered-exponential} \mathbb E\exp\bigl(\lambda(f-\mathbb Ef)\bigr) \le e^{2\lambda^2} \qquad(\lambda>0). \end{equation}
By \eqref{eq:Gq-variance}, $\Var(f)\le2$.  Every median $m_f$ of $f$ satisfies
$m_f\ge\mathbb Ef-2\sqrt2$ by Chebyshev's inequality.  It follows that
\[ \mathbb E e^{\lambda(f-m_f)_+} \le e^{2\sqrt2\lambda} \left(1+e^{2\lambda^2}\right). \]
The bound is independent of $k$, the output range, and the observable.
Exact factors and bounded-measurement limits preserve it, which proves the
first identity in \eqref{eq:phase-separation-tail}.

We next compute both invariants of the star pyramid.  Let $o$ be the common
vertex of $T_m^{\rhd}(p_1,\ldots,p_m)$ and put
\[ f_i(s)\coloneqq f(s,i)-f(o) \]
for a bounded one-sided observable $f$.  Formula
\eqref{eq:critical-star-distance} gives, for $0\le s\le t$,
\begin{equation}\label{eq:star-observable-profile} 0\le f_i(t)-f_i(s)\le2(t-s), \qquad f_i(0)=0. \end{equation}
In particular, $0\le f_i(s)\le2s$.  Since each profile is absolutely
continuous and bounded, integration by parts and
\eqref{eq:star-observable-profile} give
\[ \int_0^\infty f_i(s)^2e^{-s}\,ds =\int_0^\infty2f_i(s)f_i'(s)e^{-s}\,ds \le4\int_0^\infty f_i(s)e^{-s}\,ds. \]
Set
\[ \overline f\coloneqq\sum_{i=1}^mp_i\int_0^\infty f_i(s)e^{-s}\,ds. \]
Then $0\le\overline f\le2$, and averaging the preceding inequality over
the branches yields
\[ \Var(f)\le4\overline f-\overline f^{\,2}\le4. \]
The clipped observables $\min\{2s,L\}$ on $T_1^{\rhd}(1)$ converge in
$L^2$ to $2s$, whose variance under $e^{-s}\,ds$ is $4$.  Therefore
$\mathbf{V}^{\to}(\mathcal S_{\rhd})=4$.

For a median $m_f$ of the same star observable, the inequality
$f(s,i)\ge f(o)$ implies $m_f\ge f(o)$.  Thus, for $\lambda<1/2$,
\[ \mathbb E e^{\lambda(f-m_f)_+} \le\int_0^\infty e^{2\lambda s}e^{-s}\,ds =\frac1{1-2\lambda}. \]
For the reverse bound, take $f_L(s)\coloneqq\min\{2s,L\}$ on
$T_1^{\rhd}(1)$.  If $L>2\log2$, then its unique median is $2\log2$, and
\[ \int_{\log2}^{L/2} e^{\lambda(2s-2\log2)}e^{-s}\,ds \longrightarrow+\infty \]
as $L\to+\infty$ whenever $\lambda\ge1/2$.  This proves the second identity
in \eqref{eq:phase-separation-tail}.  The preceding uniform estimates pass
to the box closure.  Indeed, on each fixed output interval, variance is
continuous under weak convergence, and limits of convergent sequences of
medians are medians of the limiting measure.

It remains to separate the only case in which the tail exponents of a
horocone and the star agree.  On $H_{1/2}^0$, the vertical observable
$y\mapsto-2\log y$ is one-sided $1$-Lipschitz.  Its monotone clippings
converge in $L^2$ to the original observable.  If $Y\sim\chi_1$ and
$Y^2=2V$ with $V\sim\Gamma(1/2,1)$, then
\[ \Var(-2\log Y)=\Var(\log V) =\sum_{\ell=0}^\infty\frac1{(\ell+1/2)^2} =\frac{\pi^2}{2}. \]
Here the middle identity follows by differentiating the Gamma integral
twice and using the logarithmic second derivative of Euler's product for
$\Gamma$.  The last identity follows by separating the odd terms in the
Basel sum.
This proves the second inequality in \eqref{eq:phase-separation-variance}.

Now \eqref{eq:Gq-variance} and \eqref{eq:phase-separation-variance} give
$\mathcal G_q\ne\mathcal S_{\rhd}$.  Equations
\eqref{eq:Hbeta-tail-exponent} and \eqref{eq:phase-separation-tail} give
$\mathcal G_q\ne\mathcal H_\beta$ for every $\beta>0$, and they give
$\mathcal H_\beta\ne\mathcal S_{\rhd}$ when $\beta\ne1/2$.  At
$\beta=1/2$, \eqref{eq:phase-separation-variance} and $\pi^2/2>4$ give the
remaining distinction.  This completes the proof.
\end{proof}

\appendix

\section{Transfer and closure tools}\label{app:transfer-tools}

The general transfer results in \Cref{sec:general-transfer} use
\Cref{eq:pyramid-bounded-measurement,eq:measurement-convergence-criterion,eq:measurement-membership-criterion} and the
measurement-set compactness stated immediately after the first of these equations, together with the one-sided extension,
high-mass approximation, truncated-kernel transfer, and Box-closure tools collected here.  Each tool is formulated once and
is not restated for individual radial laws.

\begin{lemma}\label{lem:one-sided-extension}
Let $(X,d_X)$ be an asymmetric metric space, let $A\subset X$ be nonempty, and let $f\colon A\to[-R,R]^N$ satisfy
\begin{equation}\label{eq:additive-one-sided-condition} d_N^+(f(x),f(y))\le d_X(x,y)+\eps \qquad(x,y\in A). \end{equation}
There is a map $F\colon X\to[-R,R]^N$ satisfying \eqref{eq:bounded-measurement-condition} and
\begin{equation}\label{eq:additive-extension-error} \sup_{x\in A}\lVert F(x)-f(x)\rVert_\infty\le\eps. \end{equation}
Suppose also that $X$ carries a probability measure $\mu_X$, that $A$ is $\mu_X$-measurable, that $f$ is measurable on $(A,\mu_X|_A)$, and that a probability measure $\nu$ on $[-R,R]^N$ satisfies
\begin{equation}\label{eq:extension-submeasure} f_*(\mu_X|_A)\le\nu. \end{equation}
Then $F$ can be chosen so that
\begin{equation}\label{eq:extension-prokhorov-bound} d_{\mathrm P}(F_*\mu_X,\nu)\le\max\{\eps,1-\mu_X(A)\}. \end{equation}
In particular, this distance is at most $\eps$ when $\mu_X(A)\ge1-\eps$.
This is the prescribed-subset extension of \cite[Lemma~3.10]{one-sided-pyramids}.
The underlying infimum construction is the one-sided McShane--Whitney-type extension studied by Romaguera and Sanchis \cite{romaguera2000semi} in the quasi-metric setting.
\end{lemma}

\begin{proof}
Write $f=(f_1,\ldots,f_N)$.  For $1\le j\le N$ and $z\in X$, define
\[ F_j(z)\coloneqq\max\left\{-R,\min\left\{R,\inf_{x\in A}\{f_j(x)+d_X(x,z)\}\right\}\right\}. \]
The triangle inequality gives $F_j(z')-F_j(z)\le d_X(z,z')$.  If $y\in A$, choosing $x=y$ in the infimum gives $F_j(y)\le f_j(y)$, while \eqref{eq:additive-one-sided-condition} gives $F_j(y)\ge f_j(y)-\eps$.  Thus $F=(F_1,\ldots,F_N)$ satisfies \eqref{eq:bounded-measurement-condition} and \eqref{eq:additive-extension-error}.

Under the measure hypotheses, applying the exact inequality in both orders shows that $F$ is continuous with respect to $d_X^{\mathrm s}$ and hence Borel.  Couple $F_*(\mu_X|_A)$ and $f_*(\mu_X|_A)$ by $(F,f)_*(\mu_X|_A)$.  The positive measure $\nu-f_*(\mu_X|_A)$ and $F_*(\mu_X|_{X\setminus A})$ have the same mass and may be coupled arbitrarily.  Adding these two couplings gives a coupling of $F_*\mu_X$ and $\nu$ whose paired points are within $\eps$ on mass at least $\mu_X(A)$.  The coupling characterization of the Prokhorov distance therefore gives
\[ d_{\mathrm P}(F_*\mu_X,\nu)\le\max\{\eps,1-\mu_X(A)\}. \]
This completes the proof.
\end{proof}

\begin{lemma}\label{lem:high-mass-domination}
Let $(X_n,d_n,\mu_n)$ be qm-spaces, denoted simply by $X_n$, and let $(Y,d_Y,\nu)$, denoted simply by $Y$, be a qm-space.  Suppose that there are Borel sets $A_n\subset X_n$, Borel maps $p_n\colon A_n\to Y$, and positive numbers $\eps_n\to0$ such that
\begin{align}
\mu_n(A_n)&\longrightarrow1, \label{eq:approx-domination-mass}\\
(p_n)_*(\mu_n|_{A_n})&\longrightarrow\nu \quad\text{weakly}, \label{eq:approx-domination-measure}\\
d_Y(p_n(x),p_n(y))&\le d_n(x,y)+\eps_n \qquad(x,y\in A_n). \label{eq:approx-domination-distance}
\end{align}
Then $Y$ belongs to every subsequential pyramid limit of $(X_n)$.
This is the qm-space form of \cite[Proposition~3.11]{one-sided-pyramids}.
\end{lemma}

\begin{proof}
Pass to a subsequence for which $\Pyr(X_n)$ converges weakly to a pyramid $\mathcal P$.  Fix $N\ge1$, $R>0$, and $h_*\nu\in\mathcal M(Y;N,R)$, where $h\colon Y\to[-R,R]^N$ satisfies \eqref{eq:bounded-measurement-condition} with $d_Y$ in place of $d_X$.  On $A_n$, the map $h\circ p_n$ satisfies
\[ d_N^+(h(p_n(x)),h(p_n(y)))\le d_n(x,y)+\eps_n. \]
Apply \Cref{lem:one-sided-extension} with the probability measure
\[ (h\circ p_n)_*(\mu_n|_{A_n})+(1-\mu_n(A_n))\delta_0. \]
It gives an exact measurement $F_n\colon X_n\to[-R,R]^N$ such that
\[ d_{\mathrm P}\left((F_n)_*\mu_n,(h\circ p_n)_*(\mu_n|_{A_n})+(1-\mu_n(A_n))\delta_0\right) \le\max\{\eps_n,1-\mu_n(A_n)\}\longrightarrow0. \]
The probability measures in the second argument converge weakly to $h_*\nu$ by \eqref{eq:approx-domination-mass} and \eqref{eq:approx-domination-measure}.  Hence $(F_n)_*\mu_n\to h_*\nu$ in $d_{\mathrm P}$.  Since $(F_n)_*\mu_n\in\mathcal M(\Pyr(X_n);N,R)$, \eqref{eq:measurement-convergence-criterion} gives $h_*\nu\in\mathcal M(\mathcal P;N,R)$.  This holds for every $N$, $R$, and bounded measurement of $Y$, so \eqref{eq:measurement-membership-criterion} gives $Y\in\mathcal P$.  The subsequence was arbitrary.
This completes the proof.
\end{proof}

\begin{lemma}\label{lem:bounded-kernel-transfer}
Let $d_n$ and $d_n'$ be asymmetric metrics on the same probability space $(X_n,\mu_n)$.  Suppose that, for every $R>0$, there are Borel sets $A_{n,R}\subset X_n$ and positive numbers $\eps_{n,R}\to0$ such that
\begin{align}
\mu_n(A_{n,R})&\longrightarrow1, \label{eq:bounded-kernel-mass}\\
 \sup_{x,y\in A_{n,R}}
 \left|\min\{d_n(x,y),2R\}-\min\{d_n'(x,y),2R\}\right|
 &\le\eps_{n,R}.
 \label{eq:bounded-kernel-comparison}
\end{align}
Then the two sequences have the same weak limit points of all bounded one-sided measurement sets and hence the same subsequential pyramid limits.
\end{lemma}

\begin{proof}
Let $F\colon(X_n,d_n)\to[-R,R]^N$ satisfy \eqref{eq:bounded-measurement-condition}.  For $x,y\in A_{n,R}$, the range bound and \eqref{eq:bounded-kernel-comparison} give
\[ d_N^+(F(x),F(y)) \le\min\{d_n(x,y),2R\} \le d_n'(x,y)+\eps_{n,R}. \]
The extension in \Cref{lem:one-sided-extension} produces an exact $d_n'$-measurement which differs from $F$ by at most $\eps_{n,R}$ on $A_{n,R}$.  The Prokhorov distance of their pushforward measures tends to zero.  Interchanging $d_n$ and $d_n'$ proves equality of all bounded-measurement limit sets.
By \eqref{eq:measurement-convergence-criterion}, equality of these limit sets is equivalent to equality of the subsequential pyramid limits.
	This completes the proof.
\end{proof}

\begin{lemma}\label{lem:common-ambient-box}
Let $(K,d_K)$ be a compact metric space and let probability measures $\nu_n\longrightarrow\nu$ weakly on $K$.  Then
\[ (\supp\nu_n,d_K,\nu_n) \longrightarrow(\supp\nu,d_K,\nu) \]
in box distance.  If the spaces on the left are exact factors of qm-spaces $X_n$ and $\Pyr(X_n)\to\mathcal P$, then $(\supp\nu,d_K,\nu)$ belongs to $\mathcal P$.
\end{lemma}

\begin{proof}
This is the varying-support form of the common-ambient Prokhorov-to-box argument: the measures share a compact ambient metric space rather than a common full-support underlying space, and the final assertion passes the resulting exact factors to the pyramid limit under sequential Painlev\'e--Kuratowski convergence, which is also called weak Hausdorff convergence in mm-space theory.

Weak convergence on a compact metric space admits couplings for which the paired points are within a number tending to zero outside a set of mass tending to zero.  A relation formed by pairs at distance at most $\eps$ has distance distortion at most $2\eps$.  This proves the box convergence.

For the second assertion, every space on the left belongs to $\Pyr(X_n)$.  If the limiting support space did not belong to $\mathcal P$, \eqref{eq:weak-pyramid-outer} would give a positive lower bound for its Box distance from $\Pyr(X_n)$.  The Box approximations just obtained contradict such a bound.
	This completes the proof.
\end{proof}

\begin{theorem}\label{thm:directed-union-closure}
Let $\mathscr A$ be a nonempty family of qm-spaces directed under exact domination.  Then
\begin{equation}\label{eq:directed-union-closure} \overline{\bigcup_{X\in\mathscr A}\Pyr(X)}^{\,\Box} \end{equation}
is a pyramid.
\end{theorem}

\begin{proof}
Let $\mathcal U$ denote the union in \eqref{eq:directed-union-closure}.  It is nonempty, downward closed, and directed.

Suppose $X_n\in\mathcal U$ and $X_n\to X$ in Box distance.  Choose a generator in $\mathscr A$ above each $X_n$ and consider its associated pyramid.  The subsequential compactness statement following \eqref{eq:weak-pyramid-outer}, attributed to \cite[Corollary~5.3]{gds3-ja}, gives a subsequence converging weakly to a pyramid.  The outer condition \eqref{eq:weak-pyramid-outer} places $X$ in this limit.  If $Y\preceq X$, downward closedness of the limiting pyramid gives $Y$ in the same limit, and the inner condition \eqref{eq:weak-pyramid-inner} provides $Y_n\in\mathcal U$ with $Y_n\to Y$.  Thus the closure of $\mathcal U$ is downward closed.

For directedness, take $X,Y$ in the closure and choose $X_n,Y_n\in\mathcal U$ converging to them.  Choose generators of $\mathscr A$ above $X_n$ and $Y_n$, and then use directedness of $\mathscr A$ to choose one generator above both.  A subsequential weak limit of the corresponding associated pyramids is a pyramid.  The outer condition places $X$ and $Y$ in it, so it contains a common dominator $Z$ of $X$ and $Y$.  The inner condition supplies spaces in $\mathcal U$ converging to $Z$, and therefore $Z$ belongs to the closure.  Thus \eqref{eq:directed-union-closure} is directed.  It is Box closed by definition and hence is a pyramid.
	This completes the proof.
\end{proof}

\section{Horocones over metric bases}\label{app:horocones}

Let $(Z,d_Z,\nu)$ be an ordinary mm-space, and let $\eta$ be a Borel probability measure on $(0,+\infty)$.  Define
\begin{align}
 h_Z((z,s),(z',t))
 &\coloneqq\arcosh\left(1+\frac{d_Z(z,z')^2+(s-t)^2}{2st}\right),
 \label{eq:general-horocone-symmetric}\\
D_Z((z,s),(z',t)) &\coloneqq h_Z((z,s),(z',t))+\log\frac{s}{t}. \label{eq:general-horocone-directed}
\end{align}
Write the resulting triple as
\begin{equation}\label{eq:general-horocone-space}
 \mathsf H(Z,\eta)\coloneqq(Z\times\supp\eta,D_Z,\nu\otimes\eta).
\end{equation}
Thus, when the height measure does not have full support on the positive half-line, the product is restricted to the support
of the product measure.

\begin{theorem}[Funk horocone construction and continuity]\label{thm:general-horocone}
The space $\mathsf H(Z,\eta)$ is a qm-space and satisfies
\begin{equation}\label{eq:general-horocone-symmetrization} D_Z^{\mathrm s}((z,s),(z',t)) =h_Z((z,s),(z',t))+\left|\log\frac{s}{t}\right|. \end{equation}
If $f\colon Z\to W$ is measure preserving and $1$-Lipschitz, then $f\times\operatorname{id}$ is a measure-preserving $1$-Lipschitz map between the corresponding horocones with the same height measure.

Suppose that $Z_n\to Z$ in ordinary box distance.  If the height measure is fixed, then the directed horocones converge in asymmetric box distance.  More generally, if $\eta_n\longrightarrow\eta$ weakly on $(0,+\infty)$, then the horocones built from $(Z_n,\eta_n)$ converge to that built from $(Z,\eta)$.  If $Z$ belongs to the box closure of the factors of a family of bases, its horocone with fixed height measure belongs to the box closure of the factors of the horocones over that family.
\end{theorem}

\begin{proof}
For any three points of $Z$, their mutual distances are the side lengths of a possibly degenerate Euclidean triangle.  Realize this triangle in $\R^2$ and place the three selected heights above its vertices in the upper-half-space model of hyperbolic $3$-space.  Formula \eqref{eq:general-horocone-symmetric} is the hyperbolic distance between the lifted points.  The hyperbolic triangle inequality proves the triangle inequality for $h_Z$ on an arbitrary metric base.

We have $h_Z((z,s),(z',t))\ge|\log(s/t)|$.  Thus the endpoint potential $(z,s)\mapsto-\log s$ is $1$-Lipschitz for $h_Z$, and adding its endpoint increment gives a nonnegative quasi-pseudometric with symmetrization \eqref{eq:general-horocone-symmetrization}.  If a sequence is Cauchy for this symmetrization, its logarithmic heights are Cauchy and stay in a compact subinterval of $(0,+\infty)$.  The cosh formula then forces the base coordinates to be $d_Z$-Cauchy.  Completeness, separability, and full support follow.  Monotonicity of \eqref{eq:general-horocone-symmetric} in $d_Z$ proves functoriality.

For continuity, first fix $0<a<b<+\infty$ and couple the heights diagonally on $[a,b]$.  The function
\[ r\longmapsto2\arsinh \frac{\sqrt{r^2+(s-t)^2}}{2\sqrt{st}} \]
has derivative at most $1/a$ when $s,t\in[a,b]$.  A base relation of distortion at most $\delta$ therefore gives a directed product relation of distortion at most $\delta/a$, because the logarithmic increments agree exactly.  Its exceptional mass is bounded by the base exceptional mass plus $\eta((0,+\infty)\setminus[a,b])$.  Let the base error tend to zero and then let the interval exhaust the positive half-line.

If the height measures vary, use a Prokhorov coupling on a compact interval whose endpoints are continuity points of $\eta$.  Outside arbitrarily small mass, the paired heights lie in the interval and their difference tends to zero.  On a high-mass compact base relation, both functions in \eqref{eq:general-horocone-directed} are uniformly continuous in the base distance and the two heights.  The product relation has vanishing directed distortion.  Compact exhaustion removes both truncations.  Finally, functoriality gives exact domination before approximation, and the two continuity statements pass to the box closure of the directed base family.  This completes the proof.
\end{proof}

\subsection{Transfer proofs for general radial laws}\label{app:general-transfer-proofs}

\begin{lemma}[Uniform horospherical approximation]
\label{lem:general-horospherical-approximation}
There is a function $\delta$ such that $\delta(L)\downarrow0$ as $L\to\infty$ and, for every $n$, if
$x=\tanh\rho\,\theta$ and $y=\tanh\sigma\,\omega$ satisfy $\rho,\sigma\ge L$, then
\begin{equation}\label{eq:general-horospherical-approximation}
 \left|
 \dF(x,y)-
 \Omega_{\sqrt n/\cosh\rho,\sqrt n/\cosh\sigma}
        (\sqrt n|\theta-\omega|)
 \right|
 \le\delta(L).
\end{equation}
\end{lemma}

\begin{proof}
The endpoint-potential terms agree exactly:
\[
 \log\frac{\sqrt n/\cosh\rho}{\sqrt n/\cosh\sigma}
 =\log\cosh\sigma-\log\cosh\rho.
\]
It remains to compare the arguments of $\arcosh$.  Write the Klein argument in \Cref{eq:master-Funk} as
\[
 C_{\mathrm K}
 \coloneqq\cosh(\rho-\sigma)
  +\frac12\sinh\rho\sinh\sigma|\theta-\omega|^2.
\]
The upper-half-space argument in \Cref{eq:upper-half-Funk} is
\[
 C_{\mathrm H}
 \coloneqq\cosh(\log\cosh\rho-\log\cosh\sigma)
  +\frac12\cosh\rho\cosh\sigma|\theta-\omega|^2.
\]
Uniformly for $u\ge L$,
\[
 \begin{aligned}
 \left|\log\cosh u-(u-\log2)\right|
 &=\log(1+e^{-2u})\le e^{-2L},\\
 |\tanh u-1|&=\frac{2}{e^{2u}+1}\le2e^{-2L}.
 \end{aligned}
\]
The first estimate bounds the difference between $\log\cosh\rho-\log\cosh\sigma$ and $\rho-\sigma$ by $2e^{-2L}$.
Since $|v-w|\le c$ implies $e^{-c}\cosh w\le\cosh v\le e^c\cosh w$, the ratio of the first terms of
$C_{\mathrm H}$ and $C_{\mathrm K}$ tends uniformly to one.  The second estimate and
\[
 \sinh\rho\sinh\sigma
 =\cosh\rho\cosh\sigma\tanh\rho\tanh\sigma
\]
give the same conclusion for the second terms.  Adding the two nonnegative terms yields a function
$\varepsilon(L)\downarrow0$ such that
\[
 e^{-\varepsilon(L)}C_{\mathrm K}
 \le C_{\mathrm H}
 \le e^{\varepsilon(L)}C_{\mathrm K}.
\]

For $C\ge1$ and $c\ge1$, the hyperbolic addition formula gives
\[
 \cosh(\arcosh C+\arcosh c)
 =cC+\sqrt{C^2-1}\sqrt{c^2-1}\ge cC.
\]
Monotonicity of $\cosh$ therefore gives
\[
 \arcosh(cC)\le\arcosh C+\arcosh c.
\]
Apply this inequality to $C_{\mathrm H}\le e^{\varepsilon(L)}C_{\mathrm K}$ and
$C_{\mathrm K}\le e^{\varepsilon(L)}C_{\mathrm H}$.  Then
\[
 |\arcosh C_{\mathrm H}-\arcosh C_{\mathrm K}|
 \le\arcosh(e^{\varepsilon(L)}).
\]
Thus \Cref{eq:general-horospherical-approximation} holds with
$\delta(L)\coloneqq\arcosh(e^{\varepsilon(L)})$, and $\delta(L)\downarrow0$ as $L\to\infty$.
This completes the proof.
\end{proof}

\begin{proof}[Proof of \Cref{thm:joint-horocone-pyramid-continuity}]
From an arbitrary subsequence, take a further subsequence along which
$\mathcal H(\mathcal P_n,\eta_n)$ converges weakly to an asymmetric pyramid $\mathcal Q$.  We prove along this
subsequence that
\[
 \mathcal Q=\mathcal H(\mathcal P,\eta).
\]
Sequential compactness of the space of asymmetric pyramids will then give convergence of the full sequence.

We first prove the lower inclusion.  Fix $Z\in\mathcal P$.  The inner condition for ordinary pyramid convergence
provides $Z_n\in\mathcal P_n$ such that
\[
 Z_n\longrightarrow Z\qquad(n\to\infty)
\]
in ordinary Box distance.  Apply \Cref{thm:general-horocone} to this base convergence and to
$\eta_n\Rightarrow\eta$.  Then
\[
 \mathsf H(Z_n,\eta_n)
 \longrightarrow
 \mathsf H(Z,\eta)
 \qquad(n\to\infty)
\]
in asymmetric Box distance.  Since $\mathsf H(Z_n,\eta_n)$ belongs to
$\mathcal H(\mathcal P_n,\eta_n)$ for every $n$, the outer condition for sequential Painlev\'e--Kuratowski
convergence of pyramids gives $\mathsf H(Z,\eta)\in\mathcal Q$.  The space $Z\in\mathcal P$ was arbitrary, while
$\mathcal Q$ is downward closed and asymmetric-Box closed.  Therefore, \Cref{eq:joint-horocone-pyramid} gives
\begin{equation}\label{eq:joint-horocone-lower-inclusion}
 \mathcal H(\mathcal P,\eta)\subset\mathcal Q.
\end{equation}

For the upper inclusion, fix $N\ge1$ and $R>0$, and suppose that
\[
 \nu_n\in
 \mathcal M(\mathcal H(\mathcal P_n,\eta_n);N,R),
 \qquad
 \nu_n\Rightarrow\nu
 \qquad(n\to\infty).
\]
By the two closures in \Cref{eq:pyramid-bounded-measurement,eq:joint-horocone-pyramid}, perturbing each $\nu_n$
by less than $1/n$ in Prokhorov distance allows us to choose an ordinary mm-space $Z_n\in\mathcal P_n$ and a
$1$-Lipschitz map
\[
 F_n\colon\mathsf H(Z_n,\eta_n)\longrightarrow[-R,R]^N
\]
such that
\begin{equation}\label{eq:joint-measurement-realization}
 \nu_n=(F_n)_*(\mu_{Z_n}\otimes\eta_n).
\end{equation}
The perturbed measures still converge weakly to $\nu$, and we retain the notation $\nu_n$ for them.  Applying the
coordinatewise formula in \Cref{lem:one-sided-extension} with zero error extends $F_n$ from
$Z_n\times\supp\eta_n$ to $Z_n\times(0,+\infty)$ without changing its pushforward measure.

Choose nested compact intervals
\[
 J_\ell=[a_\ell,b_\ell]\uparrow(0,+\infty)
 \qquad(\ell\to\infty)
\]
whose endpoints all have zero $\eta$-measure.  Fix one interval $J\coloneqq J_\ell$.  For this $J$, $N$, and $R$,
use the objects $H,J^*,\mathscr K,d_{\mathscr K}$ supplied by \Cref{thm:profile-path-metric}.  For $z\in Z_n$,
define
\[
 g_{n,z}(y)\coloneqq F_n(z,y),
 \qquad y\in J^*.
\]
Applying the $1$-Lipschitz property of $F_n$ to pairs with the same base point shows that
$g_{n,z}\in\mathscr K$.  For every $z,z'\in Z_n$ and $s,t\in J^*$, it also gives
\[
 \begin{aligned}
 d_N^+(g_{n,z}(s),g_{n,z'}(t))
 &\le\Omega_{s,t}(d_{Z_n}(z,z')),\\
 d_N^+(g_{n,z'}(s),g_{n,z}(t))
 &\le\Omega_{s,t}(d_{Z_n}(z,z')).
 \end{aligned}
\]
The definition of $r_0$ in \Cref{thm:profile-path-metric} therefore yields
\begin{equation}\label{eq:joint-profile-factor-map}
 d_{\mathscr K}(g_{n,z},g_{n,z'})
 \le\min\{H,d_{Z_n}(z,z')\}
 \le d_{Z_n}(z,z').
\end{equation}
In particular, $z\mapsto g_{n,z}$ is Borel and $1$-Lipschitz.

Let $q_n^J$ be the pushforward of $\mu_{Z_n}$ under this profile map, and put
\[
 Z_n^J\coloneqq(\supp q_n^J,d_{\mathscr K},q_n^J).
\]
By \Cref{eq:joint-profile-factor-map}, we have $Z_n^J\preceq Z_n$ and hence $Z_n^J\in\mathcal P_n$.  Compactness
of $\mathscr K$ allows us, for this fixed $J$, to pass to a further subsequence such that
\[
 q_n^J\Rightarrow q^J
 \qquad(n\to\infty).
\]
Set
\[
 Z_J\coloneqq(\supp q^J,d_{\mathscr K},q^J).
\]
Weak convergence on the common compact metric space $\mathscr K$, together with a Prokhorov coupling, gives
$Z_n^J\to Z_J$ in ordinary Box distance.  Since $Z_n^J\in\mathcal P_n$, the outer condition for
$\mathcal P_n\to\mathcal P$ implies
\begin{equation}\label{eq:joint-profile-limit-base}
 Z_J\in\mathcal P.
\end{equation}

Consider the evaluation map on the common profile space,
\[
 \mathsf{ev}_J\colon\mathscr K\times J\longrightarrow[-R,R]^N,
 \qquad \mathsf{ev}_J(g,y)\coloneqq g(y).
\]
The metric $d_{\mathscr K}$ induces the uniform topology, so $\mathsf{ev}_J$ is continuous.  Moreover,
\Cref{eq:evaluation-cone-inequality} and
\[
 D_{Z_J}((g,s),(h,t))=\Omega_{s,t}(d_{\mathscr K}(g,h))
\]
show that its restriction to $Z_J\times J$ is $1$-Lipschitz on the restriction of $\mathsf H(Z_J,\eta)$ to
$Z_J\times(J\cap\supp\eta)$.  At finite $n$, the definition of the profile measure and
\Cref{eq:joint-measurement-realization} give the finite-measure relation
\begin{equation}\label{eq:joint-profile-submeasure-finite}
 (\mathsf{ev}_J)_*\bigl(q_n^J\otimes(\eta_n|_J)\bigr)
 =(F_n)_*\bigl(\mu_{Z_n}\otimes(\eta_n|_J)\bigr)
 \le\nu_n.
\end{equation}

Apply \Cref{lem:compact-profile-passage} with
$\varpi_n=q_n^J$, $\varpi=q^J$, $\xi_n=\eta_n$, $\xi=\eta$,
$\mu_n=(F_n)_*(\mu_{Z_n}\otimes\eta_n)$, $\mu=\nu$, and $E=\mathsf{ev}_J$.  Its compactness hypothesis follows
from \Cref{thm:profile-path-metric}.  The convergence $q_n^J\Rightarrow q^J$ is the subsequence just chosen,
$\eta_n\Rightarrow\eta$ is a hypothesis of the theorem, and $\nu_n\Rightarrow\nu$ is the measurement convergence
fixed above.  The endpoints of $J$ are $\eta$-null by construction, continuity of $\mathsf{ev}_J$ was just verified,
and \Cref{eq:joint-profile-submeasure-finite} supplies the required finite-stage submeasure relation.  We obtain
\begin{equation}\label{eq:joint-profile-submeasure-limit}
 (\mathsf{ev}_J)_*\bigl(q^J\otimes(\eta|_J)\bigr)
 \le\nu.
\end{equation}

Apply the coordinatewise formula in \Cref{lem:one-sided-extension}, again with zero error, to extend
$\mathsf{ev}_J|_{Z_J\times(J\cap\supp\eta)}$ to a $1$-Lipschitz map
$\widetilde{\mathsf{ev}}_J$ on all of $\mathsf H(Z_J,\eta)$.  Let
\[
 \rho_J\coloneqq(\widetilde{\mathsf{ev}}_J)_*(q^J\otimes\eta).
\]
By \Cref{eq:joint-profile-limit-base,eq:joint-horocone-pyramid}, the space $\mathsf H(Z_J,\eta)$ belongs to
$\mathcal H(\mathcal P,\eta)$.  Thus
\begin{equation}\label{eq:joint-target-measurement}
 \rho_J\in\mathcal M(\mathcal H(\mathcal P,\eta);N,R).
\end{equation}
The extension agrees with $\mathsf{ev}_J$ on $Z_J\times(J\cap\supp\eta)$, so the left-hand side of
\Cref{eq:joint-profile-submeasure-limit} is a common submeasure of $\rho_J$ and $\nu$ with mass $\eta(J)$.
Couple this common part diagonally and couple the two remaining parts, which have the same mass, arbitrarily.  The
coupling characterization of Prokhorov distance gives
\begin{equation}\label{eq:joint-exhaustion-distance}
 d_{\mathrm P}(\rho_J,\nu)
 \le2\eta((0,+\infty)\setminus J).
\end{equation}

For each $J_\ell$, the subsequence of profile measures and the limiting base $Z_{J_\ell}$ may be chosen anew.  The
measurement limit $\nu$ and the limiting height law $\eta$ do not change under these further subsequences.  The set
$\mathcal M(\mathcal H(\mathcal P,\eta);N,R)$ is compact by the property stated after
\Cref{eq:pyramid-bounded-measurement}.  Therefore,
\Cref{eq:joint-target-measurement,eq:joint-exhaustion-distance} and the exhaustion clause of
\Cref{lem:compact-profile-passage} imply
\[
 \nu\in\mathcal M(\mathcal H(\mathcal P,\eta);N,R).
\]
Since $N$ and $R$ were arbitrary, \Cref{eq:measurement-convergence-criterion} gives
\begin{equation}\label{eq:joint-horocone-upper-inclusion}
 \mathcal Q\subset\mathcal H(\mathcal P,\eta).
\end{equation}
Combining \Cref{eq:joint-horocone-lower-inclusion,eq:joint-horocone-upper-inclusion} gives
$\mathcal Q=\mathcal H(\mathcal P,\eta)$.  The original subsequence was arbitrary, so
\Cref{eq:joint-horocone-pyramid-continuity} holds for the full sequence.  This completes the proof.
\end{proof}

\begin{proof}[Proof of \Cref{thm:radial-law-transfer}]
Define the base mm-space and height law by
\[
 Z_n\coloneqq(\sqrt nS^{n-1},|\cdot|,\sigma_{n-1}),
 \qquad
 \eta_n\coloneqq\operatorname{Law}(\mathsf Y_n),
\]
and put
\begin{equation}\label{eq:comparison-space}
 \widehat X_n\coloneqq\mathsf H(Z_n,\eta_n).
\end{equation}
When $\eta_n$ does not have full support on $(0,+\infty)$, both sides are restricted to the support of the product measure,
so this is an exact identity of qm-spaces.

We first verify that the original Funk spaces and $\widehat X_n$ have the same subsequential pyramid limits.  Put
$\rho_n\coloneqq\artanh R_n$.  Then
\[
 R_n=\tanh\rho_n,
 \qquad
 \cosh\rho_n=\frac1{\sqrt{1-R_n^2}}=\frac{\sqrt n}{\mathsf Y_n}.
\]
The convergence $\eta_n=\operatorname{Law}(\mathsf Y_n)\Rightarrow\eta$ implies tightness of $(\mathsf Y_n)$.  Since
$R_n=0$ implies $\mathsf Y_n=\sqrt n$,
\[
 \mathbb P(R_n=0)
 \le\mathbb P(\mathsf Y_n\ge\sqrt n)
 \longrightarrow0
 \qquad(n\to\infty).
\]
For every fixed $L>0$, we also have
\[
 \mathbb P(\rho_n\le L)
 \le\mathbb P\left(\mathsf Y_n\ge\frac{\sqrt n}{\cosh L}\right)
 \longrightarrow0
 \qquad(n\to\infty).
\]
We may therefore choose $L_n\uparrow\infty$ such that
\begin{equation}\label{eq:joint-radial-good-set}
 \mathbb P(R_n>0,\ \rho_n\ge L_n)
 \longrightarrow1
 \qquad(n\to\infty).
\end{equation}

On the event $R_n>0$, the polar-coordinate map
\[
 R_n\theta\longmapsto(\sqrt n\theta,\mathsf Y_n)
\]
pushes $\mu_n$ forward to $\sigma_{n-1}\otimes\eta_n$ by independence of the radial and angular variables.  When two
independent samples both belong to the event in \Cref{eq:joint-radial-good-set},
\Cref{lem:general-horospherical-approximation} with $L=L_n$ bounds the difference between the Funk distance and the
directed distance of $\mathsf H(Z_n,\eta_n)$ by $\delta(L_n)$.  The exceptional mass tends to zero as
$n\to\infty$, and $\delta(L_n)\to0$.  Applying \Cref{lem:bounded-kernel-transfer} at each truncation level shows that
\begin{equation}\label{eq:joint-radial-same-subsequential-limits}
 \Pyr(\B^n,\dF,\mu_n)
 \quad\text{and}\quad
 \Pyr(\widehat X_n)
\end{equation}
have the same subsequential pyramid limits.

On the other hand, \cite[Theorem~1.1(3)]{shioya2017spheres} gives
\begin{equation}\label{eq:joint-spherical-pyramid-limit}
 \Pyr(Z_n)\longrightarrow\mathcal G_0
 \qquad(n\to\infty)
\end{equation}
weakly as ordinary pyramids.  The height-law convergence $\eta_n\Rightarrow\eta$ is
\Cref{eq:height-convergence-assumption}.  Apply
\Cref{thm:joint-horocone-pyramid-continuity} with
$\mathcal P_n=\Pyr(Z_n)$, $\mathcal P=\mathcal G_0$, and the height laws $\eta_n\Rightarrow\eta$.  We obtain
\begin{equation}\label{eq:joint-comparison-pyramid-limit}
 \mathcal H(\Pyr(Z_n),\eta_n)
 \longrightarrow
 \mathcal H(\mathcal G_0,\eta)
 \qquad(n\to\infty).
\end{equation}

We identify the left-hand side of \Cref{eq:joint-comparison-pyramid-limit}.  Since $Z_n\in\Pyr(Z_n)$,
\[
 \Pyr(\mathsf H(Z_n,\eta_n))
 \subset\mathcal H(\Pyr(Z_n),\eta_n).
\]
Conversely, fix $W\in\Pyr(Z_n)$.  The definition of the principal pyramid provides ordinary mm-spaces $W_m$ such that
$W_m\preceq Z_n$ and $W_m\to W$ as $m\to\infty$.  Functoriality in \Cref{thm:general-horocone} gives
$\mathsf H(W_m,\eta_n)\preceq\mathsf H(Z_n,\eta_n)$.  Apply the Box-continuity assertion of the same theorem to
$W_m\to W$ and the constant height-measure sequence.  Then
\[
 \mathsf H(W_m,\eta_n)\longrightarrow\mathsf H(W,\eta_n)
 \qquad(m\to\infty)
\]
in asymmetric Box distance.  It follows that
$\mathsf H(W,\eta_n)\in\Pyr(\mathsf H(Z_n,\eta_n))$.  Since $W$ was arbitrary and the right-hand side is
asymmetric-Box closed,
\begin{equation}\label{eq:principal-base-horocone-identity}
 \mathcal H(\Pyr(Z_n),\eta_n)
 =\Pyr(\mathsf H(Z_n,\eta_n))
 =\Pyr(\widehat X_n).
\end{equation}
Moreover, \Cref{eq:Gaussian-base-joint-specialization} gives
\begin{equation}\label{eq:joint-radial-target-identity}
 \mathcal H(\mathcal G_0,\eta)=\mathcal H_\eta.
\end{equation}

Combining
\Cref{eq:joint-comparison-pyramid-limit,eq:principal-base-horocone-identity,eq:joint-radial-target-identity} yields
\[
 \Pyr(\widehat X_n)\longrightarrow\mathcal H_\eta
 \qquad(n\to\infty).
\]
Finally, \Cref{eq:joint-radial-same-subsequential-limits} transfers this limit to the original Funk spaces.  This is
\Cref{eq:radial-law-pyramid-limit}.  This completes the proof.
\end{proof}

\begin{proof}[Proof of \Cref{thm:height-law-pyramid-continuity}]
We first prove that, for every Borel probability measure $\zeta$ on $(0,+\infty)$,
\begin{equation}\label{eq:Gaussian-base-joint-specialization}
 \mathcal H(\mathcal G_0,\zeta)=\mathcal H_\zeta.
\end{equation}
Each $\mathbb R_\gamma^k$ belongs to $\mathcal G_0$, so
\Cref{eq:general-horocone-pyramid,eq:joint-horocone-pyramid} gives
$\mathcal H_\zeta\subset\mathcal H(\mathcal G_0,\zeta)$.  Conversely, fix $Z\in\mathcal G_0$.  By the Box-closure
definition of the Gaussian pyramid, there are ordinary mm-spaces $W_m$ and integers $k_m$ such that
\[
 W_m\preceq\mathbb R_\gamma^{k_m},
 \qquad
 W_m\longrightarrow Z
 \qquad(m\to\infty).
\]
Functoriality in \Cref{thm:general-horocone} gives
$\mathsf H(W_m,\zeta)\preceq\mathsf H(\mathbb R_\gamma^{k_m},\zeta)$, and hence
$\mathsf H(W_m,\zeta)\in\mathcal H_\zeta$.  Apply the Box-continuity assertion of the same theorem to
$W_m\to Z$ and the constant height-measure sequence $\zeta_m=\zeta$.  Asymmetric-Box closedness of
$\mathcal H_\zeta$ then gives $\mathsf H(Z,\zeta)\in\mathcal H_\zeta$.  Since $Z\in\mathcal G_0$ was arbitrary,
the reverse inclusion follows and proves \Cref{eq:Gaussian-base-joint-specialization}.

Apply \Cref{thm:joint-horocone-pyramid-continuity} to the constant base-pyramid sequence
$\mathcal P_j=\mathcal P=\mathcal G_0$ and the height laws $\eta_j\Rightarrow\eta$.  Using
\Cref{eq:Gaussian-base-joint-specialization} on both sides gives
\Cref{eq:height-law-pyramid-continuity}.  This completes the proof.
\end{proof}

\section{Profile metrics and technical estimates}\label{app:profile-estimates}

\subsection{Radial approximation}

The radial input for the large-beta limit is a logarithmic Gamma linearization that is uniform in the ratio $\beta_n/n$.

\begin{lemma}\label{lem:radial-linearization}
Suppose that $\beta_n\to+\infty$.  The two logarithmic Gamma fluctuations
\[ \sqrt{\frac n2}\log\frac{U_n}{n/2}, \qquad \sqrt{\beta_n}\log\frac{V_n}{\beta_n} \]
converge jointly to independent standard Gaussian variables.  Moreover,
$c_n(\mathbf R_n-\bar\rho_n)$ differs in probability from
\begin{equation}\label{eq:radial-linear-part} \sqrt{\frac{\beta_n}{n+2\beta_n}} \sqrt{\frac n2}\log\frac{U_n}{n/2} -\sqrt{\frac{n}{2(n+2\beta_n)}} \sqrt{\beta_n}\log\frac{V_n}{\beta_n} \end{equation}
by a quantity tending to zero, without any restriction on the ratio $\beta_n/n$.
\end{lemma}

\begin{proof}
If $W_\alpha\sim\Gamma(\alpha,1)$ and $\alpha\to+\infty$, then
\begin{equation}\label{eq:log-gamma-clt} \sqrt\alpha\log\frac{W_\alpha}{\alpha} \overset{\mathrm{d}}{\longrightarrow} N(0,1). \end{equation}
For the standard central limit theorem and delta method underlying
\eqref{eq:log-gamma-clt}, see \cite{billingsley1995probability}.
Indeed, for fixed $t\in\R$ and all sufficiently large $\alpha$,
\[ \mathbb E\exp\left(t\frac{W_\alpha-\alpha}{\sqrt\alpha}\right) =e^{-t\sqrt\alpha}\left(1-\frac{t}{\sqrt\alpha}\right)^{-\alpha} \longrightarrow e^{t^2/2}. \]
Taylor's theorem proves the central limit theorem for $(W_\alpha-\alpha)/\sqrt\alpha$.  On the event $|W_\alpha-\alpha|\le M\sqrt\alpha$, the inequality $|\log(1+x)-x|\le x^2$ for $|x|\le1/2$ gives
\[ \left|\sqrt\alpha\log\frac{W_\alpha}{\alpha} -\frac{W_\alpha-\alpha}{\sqrt\alpha}\right| \le\frac{M^2}{\sqrt\alpha} \]
whenever $\alpha\ge4M^2$.  Tightness proves \eqref{eq:log-gamma-clt}.  Independence gives the joint convergence.

For $h(s)\coloneqq\arsinh(e^{s/2})$,
\[ h'(s)=\frac12\tanh h(s), \qquad h''(s)=\frac14\tanh h(s)\bigl(1-\tanh^2h(s)\bigr). \]
The linear term in Taylor's formula is \eqref{eq:radial-linear-part}.  The logarithmic increment
\[ \log\frac{U_n}{n/2}-\log\frac{V_n}{\beta_n} \]
tends to zero in probability.  When its absolute value is at most one, the formula for $h''$ bounds the second derivative between the two Taylor points by a universal constant times $q_n(1-q_n^2)$.  Moreover,
\begin{align*}
 \frac{c_nq_n(1-q_n^2)}{n}
 &=\frac{(2\beta_n/n)^{3/2}}
 {\sqrt n(1+2\beta_n/n)^{3/2}}
 \le\frac1{\sqrt n},\\
 \frac{c_nq_n(1-q_n^2)}{\beta_n}
 &=\frac{2\sqrt{2\beta_n/n}}
 {\sqrt n(1+2\beta_n/n)^{3/2}}
 \le\frac1{\sqrt n}.
\end{align*}
The square of the logarithmic increment is at most
\[ \frac4n\left(\sqrt{\frac n2}\log\frac{U_n}{n/2}\right)^2 +\frac2{\beta_n} \left(\sqrt{\beta_n}\log\frac{V_n}{\beta_n}\right)^2. \]
The two squared variables are tight by \eqref{eq:log-gamma-clt}.  Taylor's theorem therefore makes the remainder tend to zero after multiplication by $c_n$.  This completes the proof.
\end{proof}

\subsection{Angular approximation}

The angular input identifies both the transformed and chordal high-dimensional spheres with the canonical Gaussian pyramid.

\begin{theorem}[Transformed-sphere Gaussian convergence]\label{thm:transformed-sphere-gaussian}
Let $a_n\to+\infty$ and let $\delta_n$ be the metric in \eqref{eq:transformed-sphere-metric}.  Then
\[ \Pyr(\sqrt nS^{n-1},\delta_n,\sigma_{n-1}) \longrightarrow\mathcal G_0. \]
For every fixed $k$, the lower inclusion is realized by the first $k$-coordinate projection on Borel sets whose mass tends to one and whose additive error tends to zero.  The same pyramid convergence holds for the Euclidean chord metric.
\end{theorem}

\begin{proof}
The function $r\mapsto2a_n\arsinh(r/(2a_n))$ is increasing, concave, and vanishes at zero.  It is subadditive, so \eqref{eq:transformed-sphere-metric} is a metric.

For the lower inclusion, fix $k$.  The pushforward $(P_k)_*\sigma_{n-1}$ converges weakly to $\gamma^k$.  Restrict to
\[ A_n\coloneqq\{z\in\sqrt nS^{n-1}\colon|P_kz|\le a_n^{1/4}\}. \]
The mass of $A_n$ tends to one.  If $z,z'\in A_n$ and $|z-z'|\ge3a_n^{1/4}$, then
\[ \delta_n(z,z') \ge2a_n\arsinh\frac{3a_n^{1/4}}{2a_n} \ge2a_n^{1/4} \ge|P_kz-P_kz'| \]
for all sufficiently large $n$.  If $|z-z'|\le3a_n^{1/4}$, then
\[ 0\le |z-z'|-\delta_n(z,z') \le\frac{(3a_n^{1/4})^3}{24a_n^2}, \]
where the bound follows by integrating $1-(1+(r/(2a_n))^2)^{-1/2}$.  The last quantity tends to zero.  Thus $P_k|_{A_n}$ is $1$-Lipschitz up to a vanishing additive error.  The high-mass domination principle in \Cref{lem:high-mass-domination} puts $\gamma^k$ in every subsequential pyramid limit.  Letting $k$ vary proves the lower inclusion.

For the upper inclusion, $\delta_n(z,z')\le|z-z'|$.  Let $F_n\colon\sqrt nS^{n-1}\to[-R,R]^N$ be one-sided measurements for $\delta_n$.  Their coordinates are ordinary $1$-Lipschitz functions for the chord metric.  If $Z_n$ is standard Gaussian in $\R^n$, then
\[ \mathbb P\bigl(\bigl||Z_n|-\sqrt n\bigr|>n^{1/4}\bigr)\longrightarrow0 \]
by Chebyshev's inequality applied to $|Z_n|^2$.  On the complementary event, radial projection to $\sqrt nS^{n-1}$ has Lipschitz constant at most $\sqrt n/(\sqrt n-n^{1/4})$.  Multiplying $F_n$ by the reciprocal factor and using the extension in \Cref{lem:one-sided-extension} gives bounded Gaussian measurements whose pushforward measures approach $(F_n)_*\sigma_{n-1}$.  Every finite-dimensional standard Gaussian space belongs to $\mathcal G_0$, and the bounded-measurement criterion excludes every other limit.  Since $G_0^k$ is an exact factor of a standard Gaussian space under $(w,z)\mapsto(w/\sqrt2,z)$, $\mathcal G_0$ is exactly the canonical Gaussian pyramid.  This completes the proof.
\end{proof}

\subsection{Euclidean profiles}

A compact profile metric converts pairwise measurement inequalities into Gaussian product measurements.

\begin{lemma}\label{lem:euclidean-profile-metric}
Let $J\subset\R$ be a compact interval and let $A>0$.  Denote by $\mathscr L(J,A)$ the set of maps
$g\colon J\to[-A,A]^N$ whose coordinates are ordinary $1$-Lipschitz functions.  For $g,h\in\mathscr L(J,A)$, let
$r_0(g,h)$ be the least $r\ge0$ such that
\begin{equation}\label{eq:euclidean-profile-edge}
 \lVert g(u)-h(v)\rVert_\infty\le\sqrt{(u-v)^2+r^2}
\end{equation}
for all $u,v\in J$.  This minimum exists because the left-hand side is bounded and $J^2$ is compact.  Define
\begin{equation}\label{eq:euclidean-profile-path}
 d_{J,A}(g,h) \coloneqq
 \min\left\{2A, \inf_{g=g_0,\ldots,g_m=h}\sum_{\ell=1}^m r_0(g_{\ell-1},g_\ell)\right\}.
\end{equation}
Then $d_{J,A}$ induces the uniform topology, and $\mathscr L(J,A)$ is compact for this metric.  Moreover, the evaluation
map $(u,g)\mapsto g(u)$ is $1$-Lipschitz for the Euclidean product distance; equivalently,
\begin{equation}\label{eq:euclidean-profile-evaluation}
 \lVert g(u)-h(v)\rVert_\infty \le\sqrt{(u-v)^2+d_{J,A}(g,h)^2}.
\end{equation}
\end{lemma}

\begin{proof}
The raw path construction is a symmetric pseudometric and satisfies the triangle inequality.  Evaluation at a common parameter value in \eqref{eq:euclidean-profile-edge} gives
\[ \lVert g-h\rVert_\infty\le d_{J,A}(g,h) \]
whenever the right-hand side is below the cap, and the capped pseudometric therefore separates points.  Conversely, if $\lVert g-h\rVert_\infty\le\eps$, then
\[ \lVert g(u)-h(v)\rVert_\infty\le|u-v|+\eps
 \le\sqrt{(u-v)^2+2\operatorname{diam}(J)\eps+\eps^2}. \]
Thus $d_{J,A}(g,h)\le\sqrt{2\operatorname{diam}(J)\eps+\eps^2}$.  The two estimates prove equivalence with the uniform
topology.  Arzel\`a--Ascoli compactness gives compactness.

If $d_{J,A}(g,h)$ equals the cap, \eqref{eq:euclidean-profile-evaluation} follows from the range bound.  Otherwise, join
$(0,u)$ to $(\sum r_0,v)$ by a straight segment in the Euclidean plane and intersect it with the vertical lines at the
cumulative edge lengths of a nearly minimizing chain.  Since $J$ is convex, the intermediate vertical coordinates remain
in $J$.  Applying \eqref{eq:euclidean-profile-edge} along the chain and adding the Euclidean segment lengths proves
\eqref{eq:euclidean-profile-evaluation}.  This completes the proof.
\end{proof}

\subsection{The large-beta kernel}

Uniform comparison on radial windows of asymptotically full mass transfers the large-beta Funk geometry to a product model.

\begin{lemma}\label{lem:large-beta-kernel-comparison}
Put
\[ W_n\coloneqq c_n(\mathbf R_n-\bar\rho_n), \qquad I_n\coloneqq(-c_n\bar\rho_n,+\infty), \]
and let $\nu_n$ be the pushforward distribution of $W_n$.  In \eqref{eq:transformed-sphere-metric}, take $a_n=c_n$.  Define
\begin{align}
 D_n^{\mathrm{cmp}}((u,z),(v,z'))
 &\coloneqq\sqrt{(u-v)^2+\delta_n(z,z')^2}+q_n(v-u),
 \label{eq:large-beta-comparison-kernel}\\
\Psi_n(u,z) &\coloneqq\tanh(\bar\rho_n+u/c_n)\frac z{\sqrt n}. \label{eq:large-beta-comparison-map}
\end{align}
For every $R>0$, there are numbers $M_n\to+\infty$ such that
\begin{equation}\label{eq:large-beta-window-mass} \nu_n([-M_n,M_n])\longrightarrow1 \end{equation}
and
\begin{equation}\label{eq:large-beta-truncated-kernel}
 \begin{multlined}
 \sup_{\substack{u,v\in I_n,\ |u|,|v|\le M_n\\
 z,z'\in\sqrt nS^{n-1}}}
 \Bigl|\min\{c_n\dF(\Psi_n(u,z),\Psi_n(v,z')),2R\}\\
 {}-\min\{D_n^{\mathrm{cmp}}((u,z),(v,z')),2R\}\Bigr|
 \longrightarrow0.
 \end{multlined}
\end{equation}
The map $\Psi_n$ is measure preserving from
\[ (I_n\times\sqrt nS^{n-1},\nu_n\otimes\sigma_{n-1}) \]
to $(\B^n,\mu_{n,\beta_n})$.
\end{lemma}

\begin{proof}
We first work on a fixed radial window $|u|,|v|\le M$.  For $z=\sqrt n\theta$ and $z'=\sqrt n\eta$, the master formula gives
\begin{equation}\label{eq:large-beta-exact-cosh}
 \begin{split}
 &\cosh\dK(\Psi_n(u,z),\Psi_n(v,z'))\\
 &\quad=\cosh\frac{u-v}{c_n}
 +\frac{\sinh(\bar\rho_n+u/c_n)
 \sinh(\bar\rho_n+v/c_n)}{\sinh^2\bar\rho_n}
 \left(\cosh\frac{\delta_n(z,z')}{c_n}-1\right).
 \end{split}
\end{equation}
Since
\[ (c_nq_n)^2=\frac{2\beta_n n}{n+2\beta_n}\longrightarrow+\infty, \]
integration of the logarithmic derivative
$\cosh(\bar\rho_n+t/c_n)/(c_n\sinh(\bar\rho_n+t/c_n))$ shows that the quotient of hyperbolic sines in \eqref{eq:large-beta-exact-cosh} tends uniformly to one on the fixed window.  For every fixed $L$, Taylor's theorem applied to $\cosh$ and $\arsinh$ now gives, uniformly when $\delta_n(z,z')\le L$,
\begin{equation}\label{eq:large-beta-symmetric-uniform} c_n\dK(\Psi_n(u,z),\Psi_n(v,z')) -\sqrt{(u-v)^2+\delta_n(z,z')^2} \longrightarrow0. \end{equation}
The endpoint potential satisfies the explicit bound
\begin{equation}\label{eq:large-beta-potential-uniform} \left|c_n\left[ \phi(\bar\rho_n+v/c_n)-\phi(\bar\rho_n+u/c_n) \right]-q_n(v-u)\right| \le\frac{u^2+v^2}{2c_n}, \end{equation}
because $\phi'(\rho)=\tanh\rho$ and $|\phi''|\le1$.  Thus the directed kernels converge uniformly when the radial variables and $\delta_n$ are bounded.

Only bounded angular distances affect the truncated kernels.  The hyperbolic triangle inequality and the $1$-Lipschitz property of $\phi$ as a function of radius give
\begin{align*}
c_n\dF(\Psi_n(u,z),\Psi_n(v,z')) &\ge\delta_n(z,z')-4M,\\
D_n^{\mathrm{cmp}}((u,z),(v,z')) &\ge\delta_n(z,z')-2M.
\end{align*}
Both truncated kernels therefore equal $2R$ outside a fixed bounded angular range.  Equations \eqref{eq:large-beta-symmetric-uniform} and \eqref{eq:large-beta-potential-uniform} prove the fixed-window comparison.

The measures $\nu_n$ are tight by \Cref{thm:radial-clt}.  A diagonal choice of $M_n\to+\infty$ proves \eqref{eq:large-beta-window-mass} and \eqref{eq:large-beta-truncated-kernel}.  The map $\Psi_n$ is measure preserving by \eqref{eq:gamma-realization}.  This completes the proof.
\end{proof}

\subsection{Horocone profiles}

A path metric on bounded vertical profiles converts the horocone inequalities into a single Lipschitz evaluation estimate.

\begin{theorem}\label{thm:profile-path-metric}
Fix $N\ge1$, an output range $R>0$, and a compact height interval $J=[a,b]\Subset(0,+\infty)$.  Choose $H>0$ such that
\begin{equation}\label{eq:H-output-cutoff} \Omega_{s,t}(H)\ge2R \qquad(s,t\in J), \end{equation}
and choose $J^*=[a,B]$ to contain every height met by a hyperbolic geodesic joining $(0,s)$ to $(r,t)$ when $s,t\in J$ and $0\le r\le H$.

Let $\mathscr K$ be the set of maps $g\colon J^*\to[-R,R]^N$ satisfying
\begin{equation}\label{eq:vertical-profile-constraint} d_N^+(g(s),g(t))\le\Omega_{s,t}(0) \qquad(s,t\in J^*). \end{equation}
For $g,h\in\mathscr K$, let $r_0(g,h)$ be the least $r\ge0$ such that both
\begin{align*}
 d_N^+(g(s),h(t))&\le\Omega_{s,t}(r),\\
 d_N^+(h(s),g(t))&\le\Omega_{s,t}(r)
\end{align*}
for all $s,t\in J^*$.  Define
\begin{equation}\label{eq:profile-path-metric} d_{\mathscr K}(g,h) \coloneqq\min\left\{H, \inf_{g=g_0,\ldots,g_m=h} \sum_{\ell=1}^mr_0(g_{\ell-1},g_\ell)\right\}. \end{equation}
Then $d_{\mathscr K}$ is a compact metric whose topology is uniform convergence.  For all $g,h\in\mathscr K$ and $s,t\in J$,
\begin{equation}\label{eq:evaluation-cone-inequality} d_N^+(g(s),h(t)) \le\Omega_{s,t}(d_{\mathscr K}(g,h)). \end{equation}
\end{theorem}

\begin{proof}
The raw path construction is a symmetric pseudometric, and capping at $H$ preserves the triangle inequality.  If a chain has edge lengths $r_\ell$, then evaluation at one height gives
\[ \lVert g_{\ell-1}(s)-g_\ell(s)\rVert_\infty \le2\arsinh\frac{r_\ell}{2s} \le\frac{r_\ell}{a}. \]
Summing along the chain proves that zero path distance forces equality.  Conversely, if $\lVert g-h\rVert_\infty\le\delta$, then
\[ d_N^+(g(s),h(t))\le\Omega_{s,t}(0)+\delta. \]
The cosh formula on the compact set $J^*\times J^*$ gives a constant $C>0$, independent of $g,h$, for which the right-hand side is at most $\Omega_{s,t}(C\sqrt\delta)$.

More explicitly, after increasing $C$ if necessary, for $s,t\in J^*$ and $0\le\delta\le2R$ we have
\[ \cosh\left(\Omega_{s,t}(0)-\log\frac{s}{t}+\delta\right) \le\frac{s^2+t^2+C^2\delta}{2st} =\cosh\left(\Omega_{s,t}(C\sqrt\delta)-\log\frac{s}{t}\right). \]
Both arguments of $\cosh$ are nonnegative.  Its monotonicity therefore gives
\[ \Omega_{s,t}(C\sqrt\delta)\ge\Omega_{s,t}(0)+\delta. \]
The path metric and uniform metric therefore induce the same topology.  Compactness follows from equicontinuity in $\log y$ and the Arzel\`a--Ascoli theorem.

If $d_{\mathscr K}(g,h)=H$, \eqref{eq:evaluation-cone-inequality} follows from \eqref{eq:H-output-cutoff}.  Otherwise, take a chain of total edge length below $H$.  Join $(0,s)$ to the endpoint with horizontal coordinate equal to this total length and height $t$ by a hyperbolic geodesic.  Intersect it with vertical lines at the cumulative edge lengths.  Write $s_0=s,s_m=t$ for the endpoint heights and $s_1,\ldots,s_{m-1}$ for the intermediate heights, which lie in $J^*$ by construction.  The directed edge inequalities add along the geodesic, while their logarithmic endpoint increments telescope:
\[
 \begin{aligned}
d_N^+(g_0(s),g_m(t)) &\le\sum_{\ell=1}^m\Omega_{s_{\ell-1},s_\ell}(r_\ell)\\
 &=h_{\R}\left((0,s),\left(\sum_{\ell=1}^m r_\ell,t\right)\right)+\log\frac{s}{t}
 =\Omega_{s,t}\left(\sum_{\ell=1}^m r_\ell\right).
 \end{aligned}
\]
Taking the infimum over chains proves \eqref{eq:evaluation-cone-inequality}.  This completes the proof.
\end{proof}

\subsection{Interval exhaustion}

Weak convergence on compact height intervals passes to evaluation submeasures and then to the full height space by exhaustion.

\begin{lemma}\label{lem:compact-profile-passage}
Let $\mathscr K$ be a compact metric space, let $\varpi_n\to\varpi$ weakly on $\mathscr K$, and let $\xi_n\to\xi$ weakly on $(0,+\infty)$.  Suppose that $\mu_n\to\mu$ weakly on $[-R,R]^N$.  Let $J\Subset(0,+\infty)$ have endpoints of zero $\xi$-measure, and let $E\colon\mathscr K\times J\to[-R,R]^N$ be continuous.  If
\[ E_*(\varpi_n\otimes(\xi_n|_J))\le\mu_n \]
for every $n$, then
\[ E_*(\varpi\otimes(\xi|_J))\le\mu. \]
Suppose also that $\mathscr M$ is a compact set of probability measures on $[-R,R]^N$ and that, for every interval in a nested continuity-point exhaustion of $(0,+\infty)$, there is $\rho_J\in\mathscr M$ such that
the Prokhorov distance between $\rho_J$ and $\mu$ is at most
$2\xi((0,+\infty)\setminus J)$.
Then $\mu\in\mathscr M$.
\end{lemma}

\begin{proof}
The endpoint assumption gives $\xi_n|_J\to\xi|_J$ weakly as finite measures.  Therefore,
\[ \varpi_n\otimes(\xi_n|_J) \longrightarrow\varpi\otimes(\xi|_J). \]
For every nonnegative continuous function $\varphi$ on $[-R,R]^N$, pass to the limit in
\[ \int\varphi\circ E\,d\bigl(\varpi_n\otimes(\xi_n|_J)\bigr) \le\int\varphi\,d\mu_n \]
to obtain the submeasure relation.  Along the nested exhaustion, the displayed Prokhorov bound tends to zero.  Compactness of $\mathscr M$ then gives $\mu\in\mathscr M$.
	This completes the proof.
\end{proof}

\subsection{Spherical blocks and proximity}

Two complementary spherical estimates supply separated angular blocks of prescribed masses and nearby representatives of positive-measure sets.

\begin{theorem}\label{thm:spherical-blocks-nearby}
The following assertions hold for $(S^{n-1},\sigma_{n-1})$.
\begin{enumerate}[label=\textup{(\alph*)}]
\item Let $m\ge1$, let $p_i>0$ with $\sum_i p_i=1$, and let $\eps>0$.  There are $c>0$ and, for all sufficiently large $n$, Borel sets $B_{n,1},\ldots,B_{n,m}$ such that
\begin{align}
|\sigma_{n-1}(B_{n,i})-p_i|&<\eps, \label{eq:spherical-block-masses}\\
 \sqrt n|\theta-\eta|&\ge c
 \quad(\theta\in B_{n,i},\ \eta\in B_{n,j},\ i\ne j).
 \label{eq:spherical-block-separation}
\end{align}
The omitted mass tends to zero as $\eps\downarrow0$.
\item For every $\alpha\in(0,1)$, there is $C_\alpha<+\infty$ such that, for all sufficiently large $n$, any Borel sets $A_n,B_n\subset S^{n-1}$ of measures at least $\alpha$ contain $\theta_n\in A_n$ and $\eta_n\in B_n$ satisfying
\begin{equation}\label{eq:spherical-nearby-points} \sqrt n|\theta_n-\eta_n|\le C_\alpha. \end{equation}
\end{enumerate}
\end{theorem}

\begin{proof}
For part~\textup{(a)}, choose thresholds
\[ -\infty=t_0<t_1<\cdots<t_m=+\infty \]
so that the standard Gaussian measure of $[t_{i-1},t_i]$ is $p_i$.  Delete intervals of a sufficiently small fixed width adjacent to the finite thresholds, and let $B_{n,i}$ consist of points for which $\sqrt n\theta_1$ lies in the remaining part of $[t_{i-1},t_i]$.  Weak convergence of $\sqrt n\theta_1$ proves \eqref{eq:spherical-block-masses}.  The surviving intervals have a fixed positive separation, and coordinate projection does not increase Euclidean distance.  This proves \eqref{eq:spherical-block-separation}.

For part~\textup{(b)}, use L\'evy's spherical isoperimetric inequality \cite[Theorem~2.3]{shioya2016mmg}.  For every closed set $A\subset S^{n-1}$, every geodesic cap $C_A$ with $\sigma_{n-1}(C_A)=\sigma_{n-1}(A)$, and every $r>0$, it states
\[ \sigma_{n-1}(\{x\colon d_{S^{n-1}}(x,A)<r\})\ge\sigma_{n-1}(\{x\colon d_{S^{n-1}}(x,C_A)<r\}). \]
Inner regularity of spherical measure extends this inequality to Borel sets by approximation from within and passage to the limit in the measure of the cap. Write a cap of measure $\alpha$ as
\[ \{\theta\colon\sqrt n\theta_1\ge t_{n,\alpha}\}. \]
The sequence $t_{n,\alpha}$ is bounded by Gaussian convergence.  Expanding the cap through geodesic distance $C/\sqrt n$ decreases the first-coordinate threshold by at least $C/2$ for all sufficiently large $n$, uniformly over the bounded possible thresholds.  Indeed,
\[
 \begin{aligned}
 \sqrt n\cos\left(\arccos\frac{t_{n,\alpha}}{\sqrt n}+\frac{C}{\sqrt n}\right)
 &=t_{n,\alpha}\cos\frac{C}{\sqrt n}
 -\sqrt{n-t_{n,\alpha}^2}\sin\frac{C}{\sqrt n}\\
 &\le t_{n,\alpha}-\frac C2
 \end{aligned}
\]
for all sufficiently large $n$, uniformly because $t_{n,\alpha}$ is bounded.  Choose $C$ so that the Gaussian mass below the decreased threshold is less than $\alpha/2$.  Gaussian convergence shows that the complement of the enlarged cap has spherical measure below $\alpha$.  Spherical isoperimetry gives the same conclusion for the $C/\sqrt n$-neighborhood of $A_n$.  Since $B_n$ has measure at least $\alpha$, it meets this neighborhood.  Chord distance is at most geodesic distance, which proves \eqref{eq:spherical-nearby-points}.  This completes the proof.
\end{proof}

\subsection{Finite-star approximation}

High-mass diagonal selection and profile quantization turn compact-window approximations into measurements of finite critical stars.

\begin{lemma}\label{lem:diagonal-high-mass-selection}
Let $(X_n,d_n,\mu_n)$ and $(Y,d_Y,\nu)$ be qm-spaces, denoted simply by $X_n$ and $Y$, respectively.  For each $j$, suppose that Borel sets $A_{n,j}\subset X_n$, Borel maps $p_{n,j}\colon A_{n,j}\to Y$, numbers $\varepsilon_{n,j}\ge0$, numbers $\delta_j\downarrow0$, and finite Borel measures $\nu_j\to\nu$ satisfy
\begin{align*}
 \liminf_{n\to\infty}\mu_n(A_{n,j})&\ge1-\delta_j,\\
 \limsup_{n\to\infty}\varepsilon_{n,j}&\le\delta_j,\\
(p_{n,j})_*(\mu_n|_{A_{n,j}})&\longrightarrow\nu_j \quad(n\to\infty),
\end{align*}
and
\[ d_Y(p_{n,j}(x),p_{n,j}(y)) \le d_n(x,y)+\varepsilon_{n,j} \qquad(x,y\in A_{n,j}). \]
Then there is a sequence $j(n)\to+\infty$ for which the selected sets have mass tending to one, the selected additive errors tend to zero, and
\[ (p_{n,j(n)})_*(\mu_n|_{A_{n,j(n)}})\longrightarrow\nu. \]
\end{lemma}

\begin{proof}
Choose a metric that metrizes weak convergence of finite Borel measures on $Y$.  After passing to a subsequence of the indices $j$, assume that the distance from $\nu_j$ to $\nu$ is at most $\delta_j$.  For each $j$, choose $n_j$ so large that, for $n\ge n_j$, the mass is at least $1-2\delta_j$, the additive error is at most $2\delta_j$, and the restricted pushforward is within $\delta_j$ of $\nu_j$.  Take the integers $n_j$ increasing and set $j(n)$ equal to the largest $j$ with $n_j\le n$.  The three asserted limits follow from the triangle inequality.
	This completes the proof.
\end{proof}

\begin{theorem}\label{thm:finite-profile-quantization}
Let $I=[\delta,M]\Subset(0,+\infty)$ and let $K$ be a compact family of continuous maps $g\colon I\to[-R,R]^N$ satisfying
\begin{align}
d_N^+(g(s),g(t))&\le2(t-s)_+, \label{eq:star-profile-same}\\
d_N^+(g(s),h(t))&\le2t \quad(g\ne h) \label{eq:star-profile-cross}
\end{align}
for all $g,h\in K$ and $s,t\in I$.  Let $\varpi$ be a probability measure on $K$.  If a probability measure $\mu$ satisfies
\begin{equation}\label{eq:star-evaluation-submeasure} \bigl((g,s)\mapsto g(s)\bigr)_* \bigl(\varpi\otimes(e^{-s}\,ds|_I)\bigr)\le\mu, \end{equation}
then, for every $\eps>0$, a bounded one-sided measurement of a finite critical star has pushforward measure at Prokhorov distance at most
\begin{equation}\label{eq:star-quantization-error} \eps+2\int_{(0,+\infty)\setminus I}e^{-s}\,ds \end{equation}
from $\mu$.  Suppose that $\beta_n\to0$.  Then every weak limit of bounded measurements of $2\beta_n\widehat X_n$ belongs to $\mathcal M(\mathcal S_{\rhd};N,R)$.
\end{theorem}

\begin{proof}
Partition $K$ into finitely many nonempty Borel sets $K_1,\ldots,K_m$ of uniform diameter at most $\eps$, discard the sets of zero $\varpi$-measure, and choose $g_i\in K_i$.  Put $p_i\coloneqq\varpi(K_i)$.  On $I\times\{1,\ldots,m\}$ define
\[ f(s,i)\coloneqq g_i(s). \]
Equations \eqref{eq:star-profile-same} and \eqref{eq:star-profile-cross} are precisely the directed Lipschitz inequalities for the star distance.  \Cref{lem:one-sided-extension} extends $f$ to a bounded measurement of the full star.  Couple each $g\in K_i$ with $g_i$ and retain the radial variable.  The evaluation values differ by at most $\eps$ on $I$.  Coupling the two remaining measures arbitrarily proves \eqref{eq:star-quantization-error}.

It remains to verify that limiting profiles satisfy the hypotheses.  Let $F_n\colon2\beta_n\widehat X_n\to[-R,R]^N$ be bounded one-sided measurements whose pushforward measures tend to $\mu$.  Define
\[ g_{n,\theta}(s) \coloneqq F_n(\sqrt n\theta,e^{-s/(2\beta_n)}), \qquad s\in I. \]
For $s<t$, applying the measurement inequality to the ordered pairs $(t,s)$ and $(s,t)$ and using \eqref{eq:small-same-ray} gives, coordinatewise,
\[ 0\le g_{n,\theta,j}(t)-g_{n,\theta,j}(s)\le2(t-s) \qquad(1\le j\le N). \]
Thus \eqref{eq:small-same-ray} places the profiles in the compact class of coordinatewise nondecreasing $2$-Lipschitz maps.  Let $\varpi_n$ be their profile pushforward measures and pass to a weak limit $\varpi$.  The same-ray inequality passes directly to the support of $\varpi$.

Let $g,h\in\supp\varpi$ be distinct.  Disjoint uniform neighborhoods of $g$ and $h$ have positive $\varpi$-mass, so their inverse images among the angular variables have measures bounded below.  Write $\alpha>0$ for a common lower bound and take the neighborhoods to have uniform radius $\eps>0$.  \Cref{thm:spherical-blocks-nearby}(b) gives representatives $\theta_n$ and $\eta_n$ with scaled chord distance at most $C_\alpha$.  Monotonicity of $\Omega_{s,t}$ and \eqref{eq:small-cross}, evaluated at that constant, give the explicit comparison
\[
 \begin{aligned}
d_N^+(g(s),h(t)) &\le2\eps+d_N^+(g_{n,\theta_n}(s),g_{n,\eta_n}(t))\\
 &\le2\eps+2\beta_n\Omega_{e^{-s/(2\beta_n)},e^{-t/(2\beta_n)}}
 \bigl(\sqrt n|\theta_n-\eta_n|\bigr)\\
 &\le2\eps+2\beta_n\Omega_{e^{-s/(2\beta_n)},e^{-t/(2\beta_n)}}(C_\alpha).
 \end{aligned}
\]
Taking the upper limit gives $d_N^+(g(s),h(t))\le2\eps+2t$.  Shrinking the profile neighborhoods, equivalently letting $\eps\downarrow0$, gives
\[ d_N^+(g(s),h(t))\le2t. \]
This proves \eqref{eq:star-profile-cross}.  For a nonnegative continuous function $\varphi$ on $[-R,R]^N$, independence of the angular and radial variables gives the product measures and the finite-$n$ identity
\[
 \begin{aligned}
&\int\varphi\,d\bigl((g,s)\mapsto g(s)\bigr)_* \bigl(\varpi_n\otimes((S_n)_*\mathbb P|_I)\bigr)\\
 &\qquad=\mathbb E\left[\mathbf 1_{\{S_n\in I\}}
 \varphi\bigl(F_n(\sqrt n\Theta_n,e^{-S_n/(2\beta_n)})\bigr)\right]\\
&\qquad\le\int\varphi\,d(F_n)_* \bigl(\sigma_{n-1}\otimes(Y_n)_*\mathbb P\bigr),
 \end{aligned}
\]
where $\Theta_n$ has distribution $\sigma_{n-1}$.  The endpoints of $I$ have zero mass under the exponential limit.  Therefore, \Cref{thm:small-beta-radial} gives convergence of the restricted height measures, and continuity of evaluation yields \eqref{eq:star-evaluation-submeasure}.  Apply the first part with $I_j=[1/j,j]$ and $\eps=1/j$.  The quantization bound satisfies
\[ \frac1j+2\int_{(0,+\infty)\setminus I_j}e^{-s}\,ds\longrightarrow0. \]
The resulting pushforward measures belong to the compact set $\mathcal M(\mathcal S_{\rhd};N,R)$ and converge to $\mu$, so compactness proves the final assertion.  This completes the proof.
\end{proof}

\section{Marked Gaussian products and high-horocone tangent limits}\label{app:marked-product-proofs}

Write the symmetric product metric under the square root as $d_{2,Z}$, and put $\psi_q(u,z)\coloneqq qu$.  It follows from
\Cref{eq:Dq-general} that
\begin{equation}\label{eq:Dq-potential}
 D_{q,Z}((u,z),(v,z'))=d_{2,Z}((u,z),(v,z'))+\psi_q(v,z')-\psi_q(u,z).
\end{equation}
Therefore,
\begin{equation}\label{eq:Dq-observables}
 f\in\Lipplus(\mathsf G_q(Z,\vartheta))
 \quad\Longleftrightarrow\quad
 f-\psi_q\in\operatorname{Lip}_1(\R\times_2 Z).
\end{equation}
Subtracting the potential increment in \Cref{eq:Dq-potential} from the one-sided Lipschitz inequality and interchanging the
two points gives the implication from left to right.  The converse follows by restoring the potential increment.

\begin{lemma}[Marked-product transfer]
\label{lem:marked-product-transfer}
Let $(Z_n)$ be ordinary mm-spaces whose associated pyramids converge weakly to an ordinary pyramid $\mathcal P$, and
suppose that probability measures on $\R$ converge weakly as $\vartheta_n\Rightarrow\vartheta$.  Then the ordinary product
spaces satisfy
\begin{equation}\label{eq:ordinary-marked-product-limit}
 \Pyr\bigl((\R,|\cdot|,\vartheta_n)\times_2 Z_n\bigr)
 \longrightarrow
 \overline{\bigcup_{Z\in\mathcal P}
 \Pyr((\R,|\cdot|,\vartheta)\times_2 Z)}^{\,\Box}.
\end{equation}
\end{lemma}

\begin{proof}
Apply \Cref{prop:endpoint-potential-transport} with $\mathcal P_n=\Pyr(Z_n)$ and $q_n=q=0$.  In this case,
\Cref{eq:Dq-general} gives $D_{0,Z}=d_{2,Z}$, and
\[
 \mathsf G_0(\Pyr(Z_n),\vartheta_n)
 =\Pyr\bigl((\R,|\cdot|,\vartheta_n)\times_2 Z_n\bigr).
\]
Indeed, $Z_n\in\Pyr(Z_n)$ gives the inclusion from right to left, while functoriality under the product of an exact
domination map with the identity on $\R$, followed by Box closure, gives the reverse inclusion.  Moreover,
\Cref{eq:marked-pyramid-operation} identifies $\mathsf G_0(\mathcal P,\vartheta)$ with the right-hand side of
\Cref{eq:ordinary-marked-product-limit}.  The conclusion of \Cref{prop:endpoint-potential-transport} therefore gives
\Cref{eq:ordinary-marked-product-limit}.  This completes the proof.
\end{proof}

\begin{proof}[Proof of \Cref{prop:endpoint-potential-transport}]
We prove the two Painlev\'e--Kuratowski inclusions through bounded measurements.  For the lower inclusion, fix
$Z\in\mathcal P$.  The inner condition for $\mathcal P_n\to\mathcal P$ supplies ordinary mm-spaces
$W_n\in\mathcal P_n$ such that $W_n\to Z$ in ordinary Box distance.  We claim that
\begin{equation}\label{eq:directed-generator-box-limit}
 \mathsf G_{q_n}(W_n,\vartheta_n)
 \longrightarrow
 \mathsf G_q(Z,\vartheta)
\end{equation}
in qm-Box distance.

Fix $\varepsilon>0$.  Choose a compact mark interval $J=[-M,M]$ with continuity-point endpoints such that
$\vartheta(\R\setminus J)<\varepsilon$.  For all large $n$, weak convergence gives
$\vartheta_n(\R\setminus J)<2\varepsilon$ and a Prokhorov coupling of $\vartheta_n$ and $\vartheta$ under which the marks
differ by more than $\varepsilon$ on mass at most $\varepsilon$.  For the bases, choose a high-mass Box relation between
$W_n$ and $Z$, and if necessary restrict each base to a compact set of mass greater than $1-\varepsilon$.  The base
distances are bounded on the resulting relation, and their distortion tends to zero.

On these compact restrictions, the function
\[
 (r,u,v,a)\longmapsto
 \sqrt{(u-v)^2+r^2}+a(v-u)
\]
is uniformly continuous.  The base-distance distortion tends to zero, the coupled marks approach each other, and
$q_n\to q$.  Hence the distortion of the qm-distances in \Cref{eq:Dq-general} tends to zero.  The exceptional mass tends
to zero by first taking $n\to\infty$ and then $M\to\infty$ and $\varepsilon\downarrow0$.  This proves
\Cref{eq:directed-generator-box-limit}.  Every finite generator of
$\mathsf G_q(\mathcal P,\vartheta)$ therefore belongs to every subsequential pyramid limit, proving the lower inclusion.

For the upper inclusion, fix $N\ge1$ and $R>0$, and suppose that
\[
 \nu_n\in
 \mathcal M(\mathsf G_{q_n}(\mathcal P_n,\vartheta_n);N,R),
 \qquad \nu_n\Rightarrow\nu.
\]
By the definition and closure of the measurement set, for every $n$ there are an ordinary mm-space
$Z_n\in\mathcal P_n$ and a $1$-Lipschitz map
\[
 F_n\colon\mathsf G_{q_n}(Z_n,\vartheta_n)\longrightarrow[-R,R]^N
\]
such that
\[
 d_{\mathrm P}\bigl((F_n)_*(\vartheta_n\otimes\mu_{Z_n}),\nu_n\bigr)<\frac1n.
\]
Replacing $\nu_n$ by this pushforward does not change its limit, so assume below that
\[
 \nu_n=(F_n)_*(\vartheta_n\otimes\mu_{Z_n}).
\]
The coordinatewise extension formula in \Cref{lem:one-sided-extension} is independent of the measure.  Extend $F_n$ from
$\supp\vartheta_n\times Z_n$ to all of $\R\times Z_n$ without changing its pushforward measure.

Fix a compact interval $J=[-M,M]$ whose endpoints are continuity points of $\vartheta$, and put
$\mathbf1_N\coloneqq(1,\ldots,1)\in\R^N$.  For $z\in Z_n$, define
\begin{equation}\label{eq:directed-transformed-profile}
 g_{n,z}(u)\coloneqq F_n(u,z)-q_nu\,\mathbf1_N,
 \qquad u\in J.
\end{equation}
Applying \Cref{eq:Dq-potential} in both orders gives, for $u,v\in J$ and $z,z'\in Z_n$,
\begin{equation}\label{eq:directed-profile-cross-bound}
 \|g_{n,z}(u)-g_{n,z'}(v)\|_\infty
 \le\sqrt{(u-v)^2+d_{Z_n}(z,z')^2}.
\end{equation}
Indeed, subtracting the potential increment $q_n(v-u)$ from the $1$-Lipschitz inequality for each coordinate of $F_n$
bounds the difference in one direction; interchanging the points gives the reverse bound.  Since $F_n$ takes values in
$[-R,R]^N$ and $|q_nu|\le M$, the profiles take values in the fixed cube $[-R-M,R+M]^N$.

Apply \Cref{lem:euclidean-profile-metric} with $A=R+M$.  The space $\mathscr L(J,R+M)$ is compact, and its metric
$d_{J,R+M}$ satisfies
\begin{equation}\label{eq:directed-profile-evaluation}
 \|g(u)-h(v)\|_\infty
 \le\sqrt{(u-v)^2+d_{J,R+M}(g,h)^2}.
\end{equation}
We pass the profile measures to the limit on this common compact profile space.

Equation \eqref{eq:directed-profile-cross-bound} shows that $z\mapsto g_{n,z}$ is $1$-Lipschitz.  Let $\kappa_n^J$ be the
pushforward of the base measure $\mu_{Z_n}$ under this map, and put
\[
 V_n^J\coloneqq(\supp\kappa_n^J,d_{J,R+M},\kappa_n^J).
\]
The profile map is an exact domination map from $Z_n$ to $V_n^J$, so $V_n^J\preceq Z_n$ and $V_n^J\in\mathcal P_n$.

Compactness of $\mathscr L(J,R+M)$ permits passage to a subsequence such that $\kappa_n^J\Rightarrow\kappa^J$.  Put
\[
 V^J\coloneqq(\supp\kappa^J,d_{J,R+M},\kappa^J).
\]
A Prokhorov coupling on the common compact metric space gives $V_n^J\to V^J$ in ordinary Box distance.  The outer condition
for $\mathcal P_n\to\mathcal P$ then gives $V^J\in\mathcal P$.

On the subset $J\times V^J$, define
\begin{equation}\label{eq:directed-profile-limit-map}
 E_M^0(u,g)\coloneqq
 \operatorname{clip}_{[-R,R]^N}
 \bigl(g(u)+qu\,\mathbf1_N\bigr).
\end{equation}
By \Cref{eq:directed-profile-evaluation}, each coordinate of the unclipped map after subtracting
$qu\,\mathbf1_N$ is $1$-Lipschitz for the ordinary $\ell_2$ product metric.  Hence
\Cref{eq:Dq-observables} shows that each coordinate before clipping is one-sided $1$-Lipschitz for $D_{q,V^J}$.
Coordinatewise clipping on $\R$ is nondecreasing and $1$-Lipschitz, so it preserves this inequality.  Thus $E_M^0$ is a
$1$-Lipschitz map from the restricted qm-space to $[-R,R]^N$.  Apply the coordinatewise formula in
\Cref{lem:one-sided-extension} directly to $E_M^0$.  Since the restriction already satisfies the exact $1$-Lipschitz
inequality, the extension agrees with it there and gives a $1$-Lipschitz map
\[
 E_M\colon\mathsf G_q(V^J,\vartheta)\longrightarrow[-R,R]^N.
\]

The profile variable depends only on the base point and is independent of the mark.  On the common compact space
$J\times\mathscr L(J,R+M)$, the maps
\[
 (u,g)\longmapsto
 \operatorname{clip}_{[-R,R]^N}
 \bigl(g(u)+q_nu\,\mathbf1_N\bigr)
\]
converge uniformly to $E_M^0$; their $\ell^\infty$ distance is at most $M|q_n-q|$.  The continuity-point assumption on the
endpoints of $J$ gives $\vartheta_n|_J\Rightarrow\vartheta|_J$, while $\kappa_n^J\Rightarrow\kappa^J$.  Weak convergence of the
product finite measures and uniform convergence of the maps therefore give
\[
 (F_n)_*(\vartheta_n|_J\otimes\mu_{Z_n})
 \Longrightarrow
 (E_M^0)_*(\vartheta|_J\otimes\kappa^J).
\]
The measure on the left is a submeasure of $\nu_n$, and $\nu_n\Rightarrow\nu$.  Passing to the limit against every
nonnegative continuous function yields
\[
 (E_M^0)_*(\vartheta|_J\otimes\kappa^J)\le\nu.
\]
Since $E_M=E_M^0$ on $J\times V^J$, the same finite measure is a submeasure of
$(E_M)_*(\vartheta\otimes\kappa^J)$, and its mass is $\vartheta(J)$.

Couple the common part of $\nu$ and $(E_M)_*(\vartheta\otimes\kappa^J)$ diagonally and the remaining equal-mass parts
arbitrarily.  Then
\begin{equation}\label{eq:directed-window-error}
 d_{\mathrm P}
 \left(\nu,(E_M)_*(\vartheta\otimes\kappa^J)\right)
 \le1-\vartheta(J).
\end{equation}
Because $V^J\in\mathcal P$, the right-hand pushforward belongs to
$\mathcal M(\mathsf G_q(\mathcal P,\vartheta);N,R)$.

Exhaust $\R$ by intervals $J=[-M,M]$ with continuity-point endpoints.  For each $J$, the subsequence producing the profile
limit and the base $V^J$ may be chosen anew.  Every resulting pushforward belongs to the fixed compact set
$\mathcal M(\mathsf G_q(\mathcal P,\vartheta);N,R)$, and the right-hand side of
\Cref{eq:directed-window-error} tends to zero.  Closedness gives
\[
 \nu\in\mathcal M(\mathsf G_q(\mathcal P,\vartheta);N,R),
\]
which proves the upper inclusion.  The two inclusions hold for every $N$ and $R$, so
\Cref{eq:measurement-convergence-criterion} proves \Cref{eq:directed-marked-product-limit}.
This completes the proof.
\end{proof}

\begin{proof}[Proof of \Cref{thm:horocone-tangent}]
Put $\vartheta_j\coloneqq\operatorname{Law}(\zeta_j)$.  Fix an ordinary mm-space $Z$ with distance $d_Z$.  Under the
change of variables $y=h_j-\xi$, the qm-distance on the horocone over $Z$, multiplied by $h_j$, becomes
\begin{equation}\label{eq:scaled-horocone-mark-kernel}
\begin{split}
 K_{j,Z}((\xi,z),(\zeta,z'))
 &\coloneqq 2h_j\arsinh\left(
   \frac{\sqrt{d_Z(z,z')^2+(\xi-\zeta)^2}}
        {2\sqrt{(h_j-\xi)(h_j-\zeta)}}
  \right)\\
 &\quad+h_j\log\frac{h_j-\xi}{h_j-\zeta}.
\end{split}
\end{equation}
Indeed, this follows by substituting $y=h_j-\xi$ and $y'=h_j-\zeta$ into
\Cref{eq:general-horocone-symmetric,eq:general-horocone-directed} and multiplying by $h_j$.

Consider bounded $\xi$, $\zeta$, and $r=d_Z(z,z')$.  Since $h_j\to\infty$,
\[
 \frac{\sqrt{r^2+(\xi-\zeta)^2}}
      {2\sqrt{(h_j-\xi)(h_j-\zeta)}}\longrightarrow0
\]
locally uniformly.  Since $\arsinh x/x\to1$ as $x\to0$,
\begin{equation}\label{eq:tangent-symmetric-kernel}
 2h_j\arsinh\left(
  \frac{\sqrt{r^2+(\xi-\zeta)^2}}
       {2\sqrt{(h_j-\xi)(h_j-\zeta)}}
 \right)
 \longrightarrow\sqrt{r^2+(\xi-\zeta)^2}
\end{equation}
locally uniformly.  For the potential term, if $|a|\le M$ and $h\ge2M$, then
\[
 \left|h\log(1-a/h)+a\right|
 \le\frac{a^2}{h(1-|a|/h)}
 \le\frac{2M^2}{h}.
\]
Apply this estimate with $a=\xi$ and $a=\zeta$ and take the difference.  Then
\begin{equation}\label{eq:tangent-potential-kernel}
 h_j\log\frac{h_j-\xi}{h_j-\zeta}
 \longrightarrow\zeta-\xi
\end{equation}
locally uniformly.  The limiting kernel in \Cref{eq:scaled-horocone-mark-kernel} is therefore
\[
 \sqrt{r^2+(\xi-\zeta)^2}+\zeta-\xi
 =D_{1,Z}((\xi,z),(\zeta,z')).
\]

We show that this local uniform convergence suffices to compare every bounded measurement.  Fix $N\ge1$ and $R>0$.
Choose continuity points $\pm M$ of $\vartheta$ such that $\vartheta([-M,M])$ is arbitrarily close to one.  Since
$\vartheta_j\Rightarrow\vartheta$, the same high-mass conclusion holds for $\vartheta_j$ for all large $j$.  Write
$K_j(r,\xi,\zeta)$ for the scalar kernel obtained from \Cref{eq:scaled-horocone-mark-kernel} by replacing
$d_Z(z,z')$ with $r\ge0$, and put
\[
 K(r,\xi,\zeta)\coloneqq\sqrt{r^2+(\xi-\zeta)^2}+\zeta-\xi.
\]
We claim that
\begin{equation}\label{eq:tangent-truncated-uniform}
\begin{split}
 \sup_{\substack{|\xi|,|\zeta|\le M\\r\ge0}}
 \bigl|&\min\{K_j(r,\xi,\zeta),2R\}\\
       &-\min\{K(r,\xi,\zeta),2R\}\bigr|
 \longrightarrow0.
\end{split}
\end{equation}

By the local uniform convergence in \Cref{eq:tangent-potential-kernel}, the potential term of $K_j$ is bounded in absolute
value by a constant $B_M$ for $|\xi|,|\zeta|\le M$ and all large $j$.  If $K_j(r,\xi,\zeta)\le2R$, its symmetric term is at
most $2R+B_M$.
Monotonicity of $\arsinh$ gives
\[
 \sqrt{r^2+(\xi-\zeta)^2}
 \le2(h_j+M)\sinh\left(\frac{2R+B_M}{2h_j}\right)
 \le C_{R,M}
\]
for a constant $C_{R,M}$.  The last bound follows from $h_j\to\infty$ and $\sinh x/x\to1$.  If instead
$K(r,\xi,\zeta)\le2R$, then
\[
 \sqrt{r^2+(\xi-\zeta)^2}
 \le2R+|\zeta-\xi|\le2R+2M.
\]
Thus, whenever at least one truncated kernel is below $2R$, the variable $r$ lies in a compact interval depending only on
$R$ and $M$.  If both truncated kernels equal $2R$, their difference is zero.  Applying the local uniform convergences in
\Cref{eq:tangent-symmetric-kernel,eq:tangent-potential-kernel} on this compact interval proves
\Cref{eq:tangent-truncated-uniform}.

For every ordinary base $Z$, the map
\[
 (z,y)\longmapsto(h_j-y,z)
\]
pushes the height law $\eta_j$ to the mark law $\vartheta_j$ and preserves the base measure, hence preserves the product
measure.  On $\{|h_j-y|\le M\}$, \Cref{eq:tangent-truncated-uniform} compares the scaled horocone distance and the
asymmetric marked-product distance uniformly in the base $Z$.  Applying \Cref{lem:bounded-kernel-transfer} to each
generator gives, for every fixed $N$ and $R$,
\begin{equation}\label{eq:tangent-measurement-comparison}
\begin{split}
 (d_{\mathrm P})_H\bigl(&
  \mathcal M(h_j\mathcal H_{\eta_j};N,R),\\
 &\mathcal M(\mathsf G_1(\mathcal G_0,\vartheta_j);N,R)
 \bigr)\longrightarrow0.
\end{split}
\end{equation}
More explicitly, first let $j\to\infty$ with $M$ fixed; the truncated-distance error on the window then tends to zero.
Next let $M\to\infty$; the mass outside the mark window tends uniformly to zero.  The error in
\Cref{eq:tangent-truncated-uniform} is independent of the base in the Gaussian pyramid, so the same estimate survives
passage to exact factors and qm-Box closure.  This proves \Cref{eq:tangent-measurement-comparison} from the generatorwise
comparison.

Finally apply \Cref{prop:endpoint-potential-transport} with the constant base pyramids
$\mathcal P_j=\mathcal P=\mathcal G_0$ and $q_j=q=1$.  Since $\vartheta_j\Rightarrow\vartheta$,
\[
 \mathsf G_1(\mathcal G_0,\vartheta_j)
 \longrightarrow\mathcal G_{1,\vartheta}.
\]
Together with \Cref{eq:tangent-measurement-comparison}, this identifies all bounded-measurement limits.  Since $N$ and $R$
were arbitrary, \Cref{eq:measurement-convergence-criterion} proves
\Cref{eq:horocone-tangent-limit}.  This completes the proof.
\end{proof}

\clearpage

\section*{Acknowledgments}

The author would like to thank Professor Takashi Shioya for many helpful suggestions and guidance. The author used Claude, GPT-5.5, and GPT-5.6-series Codex models as AI-assisted tools in preparing this manuscript.  The author reviewed and revised the mathematical content and takes full responsibility for the final manuscript.


\begin{thebibliography}{10}

\bibitem{billingsley1995probability}
P.~Billingsley.
\newblock {\em Probability and Measure}.
\newblock Wiley Series in Probability and Mathematical Statistics. John Wiley
  \& Sons, Inc., New York, third edition, 1995.

\bibitem{esaki-kazukawa-mitsuishi2024invariants}
S.~Esaki, D.~Kazukawa, and A.~Mitsuishi.
\newblock Invariants for {G}romov's pyramids and their applications.
\newblock {\em Adv. Math.}, 442:109583, 2024.

\bibitem{esaki-kazukawa-mitsuishi2024cones}
S.~Esaki, D.~Kazukawa, and A.~Mitsuishi.
\newblock Convergence of cones of metric measure spaces and its application to
  {C}auchy distribution.
\newblock {\em Int. Math. Res. Not. IMRN}, 2025(18):rnaf292, 2025.
\newblock arXiv:2402.14331.

\bibitem{kristaly-ohta-zhao2025analysis}
A.~Krist{\'a}ly, S.~Ohta, and W.~Zhao.
\newblock Analysis on asymmetric metric measure spaces: {$q$}-heat flow,
  {$q$}-{Laplacian} and {Sobolev} spaces, 2025.

\bibitem{kristaly-zhao2022geometry}
A.~Krist{\'a}ly and W.~Zhao.
\newblock On the geometry of irreversible metric-measure spaces: Convergence,
  stability and analytic aspects.
\newblock {\em Journal de Math{\'e}matiques Pures et Appliqu{\'e}es (9)},
  158:216--292, 2022.

\bibitem{ohta2021comparison}
S.~Ohta.
\newblock {\em Comparison Finsler Geometry}.
\newblock Springer Monographs in Mathematics. Springer, Cham, 2021.

\bibitem{romaguera2000semi}
S.~Romaguera and M.~Sanchis.
\newblock Semi-{Lipschitz} functions and best approximation in quasi-metric
  spaces.
\newblock {\em J. Approx. Theory}, 103(2):292--301, 2000.

\bibitem{shioya2016mmg}
T.~Shioya.
\newblock {\em Metric measure geometry}, volume~25 of {\em IRMA Lectures in
  Mathematics and Theoretical Physics}.
\newblock EMS Publishing House, Z\"urich, 2016.
\newblock Gromov's theory of convergence and concentration of metrics and
  measures.

\bibitem{shioya2017spheres}
T.~Shioya.
\newblock Metric measure limits of spheres and complex projective spaces.
\newblock In {\em Measure Theory in Non-Smooth Spaces}, Partial Differential
  Equations and Measure Theory, pages 261--287. De Gruyter Open, Warsaw, 2017.

\bibitem{shioya2022sugaku}
T.~Shioya.
\newblock Metric measure geometry: an approach to high-dimensional and
  infinite-dimensional spaces.
\newblock {\em Sugaku Expositions}, 35(2):221--241, 2022.
\newblock Translation of S\={u}gaku \textbf{71} (2019), no.~2, 159--177.

\bibitem{stojmirovic2004quasi}
A.~Stojmirovi{\'c}.
\newblock Quasi-metric spaces with measure.
\newblock {\em Topology Proceedings}, 28(2):655--671, 2004.

\bibitem{one-sided-pyramids}
S.~Yokota.
\newblock Extensions of one-sided box geometry and pyramid invariants to
  gd-sets and qm-spaces.
\newblock preprint, arXiv:2608.12749, 2026.

\bibitem{yokota-poincare-beta-balls}
S.~Yokota.
\newblock Poincar\'e beta balls: Radial laws, shell transforms, and phase
  diagram.
\newblock preprint, arXiv:2607.26979, 2026.

\bibitem{gds3-ja}
S.~Yokota.
\newblock Pyramidal compactification of asymmetric metric measure spaces via
  adjoint transport.
\newblock preprint, arXiv:2608.01145, 2026.

\end{thebibliography}
\end{document}